\documentclass[a4paper,11pt]{article}
\usepackage{amsmath,amssymb,amsthm,graphicx,subfig}
\usepackage{hyperref,url,enumitem,bm,esint,makecell}
\usepackage{color}
\usepackage[margin=30mm]{geometry}

\newcommand{\pd}{\partial}
\newcommand{\eps}{\varepsilon}
\newcommand{\Lap}{\Delta}
\renewcommand{\div}{\mathrm{div}\,}
\newcommand{\R}{\mathbb{R}}
\newcommand{\pdnu}{\pd_{\bm{n}}}
\newcommand{\inn}[2]{\langle #1, #2 \rangle}
\newcommand{\no}[1]{\| #1 \|}
\newcommand{\N}{\mathbb{N}}
\newcommand{\dx}{\, \mathrm{dx}}
\newcommand{\dy}{\, \mathrm{dy}}

\newcommand{\dt}{\, \mathrm{dt}}
\newcommand{\vp}{\varphi}
\newcommand{\Th}{\mathcal{T}_h}
\newcommand{\Sh}{\mathcal{S}_h}
\newcommand{\Ih}{\mathrm{I}_h}
\newcommand{\bS}{\mathbb{S}}
\newcommand{\bM}{\mathbb{M}}
\newcommand{\mean}[1]{\langle #1 \rangle_\Omega}
\newcommand{\V}{\mathcal{V}}
\newcommand{\tc}[1]{#1}

\theoremstyle{plain}
\newtheorem{theorem}{Theorem}[section]

\newtheorem{remark}{Remark}[section]

\numberwithin{equation}{section}

\title{Stability and convergence of relaxed scalar auxiliary variable schemes for Cahn--Hilliard systems with bounded mass source}
\author{Kei Fong Lam \footnotemark[1] \and Ru Wang \footnotemark[1]}
\date{ }

\begin{document}
\maketitle

\renewcommand{\thefootnote}{\fnsymbol{footnote}}
\footnotetext[1]{Department of Mathematics, Hong Kong Baptist University, Kowloon Tong, Hong Kong ({\tt $\{$akflam,wangru\_22$\}$@math.hkbu.edu.hk}).}

\begin{abstract} 
The scalar auxiliary variable (SAV) approach of Shen et al.~(2018), which presents a novel way to discretize a large class of gradient flows, has been extended and improved by many authors for general dissipative systems. In this work we consider a Cahn--Hilliard system with mass source that, for image processing and biological applications, may not admit a dissipative structure involving the Ginzburg--Landau energy. Hence, compared to previous works, the stability of SAV-discrete solutions for such systems is not immediate.  We establish, with a bounded mass source, stability and convergence of time discrete solutions for a first-order relaxed SAV scheme in the sense of Jiang et al.~(2022), and apply our ideas to Cahn--Hilliard systems appearing in diblock co-polymer phase separation, tumor growth, image inpainting and segmentation.
\end{abstract}


\section{Introduction}
We are interested in the stability and convergence of numerical schemes based on the scalar auxiliary variable (SAV) approach \cite{SAV1,SAV2} for Cahn--Hilliard systems of the form:
\begin{subequations}\label{CH}
\begin{alignat}{3}
\pd_t \vp & = \div (m(\vp) \nabla \mu) + f  \quad && \text{ in } \Omega \times (0,T),  \label{CH:1} \\
\mu & = - \eps \Lap \vp + \eps^{-1} F'(\vp) - g \quad && \text{ in } \Omega \times (0,T) \label{CH:2}.
\end{alignat}
\end{subequations}
In \eqref{CH}, $\vp$ and $\mu$ denote the phase field variable and its corresponding chemical potential, $\eps > 0$ is a small fixed parameter, $m(\vp)$ is a phase-dependent mobility function, $F$ is a potential function with derivative $F'$, and $f, g$ are source terms that can be prescribed or may depend on $\vp$. We furnish \eqref{CH:1}-\eqref{CH:2} with homogeneous Neumann conditions $\pdnu \vp = \pdnu \mu = 0$ on $\pd \Omega \times (0,T)$, and initial condition $\vp(0) = \vp_0$ in $\Omega$. The above system admits the following energy identity:
\begin{align}\label{CH:energy}
\frac{d}{dt} E_\eps(\vp) + \int_\Omega m(\vp) |\nabla \mu|^2 \dx = \int_\Omega f \mu + g \pd_t \vp \dx
\end{align}
involving the Ginzburg--Landau functional $
E_\eps(\vp) = \int_\Omega \frac{\eps}{2} |\nabla \vp|^2 + \frac{1}{\eps} F(\vp) \dx$. 

Under suitable choices of $f$ and $g$, such types of Cahn--Hilliard systems have seen recent applications in image inpainting \cite{Bert2}, image segmentation \cite{Yang}, phase separation of diblock copolymer \cite{Gior,Oono,Ohta}, evolution of brain metabolites concentration \cite{LiMiran}, and tumor growth \cite{Crisbook,GLSS,Wise}. Thus, efficient, stable and convergent numerical schemes for systems such as \eqref{CH} are necessary to provide meaningful simulations to these areas of applications. We refer to \cite{Aris,Fakih} for some previous numerical contributions for \eqref{CH} with $g = 0$. 

In the absence of the source terms $f$ and $g$, and the mobility $m(\vp)$ set to 1, the resulting Cahn--Hilliard equation can be viewed as a $H^{-1}$-gradient flow of the Ginzburg--Landau functional $E_\eps(\vp)$. The only nonlinearity is the derivative term $F'(\vp)$, where a typical choice for the potential function $F$ is the quartic polynomial $F(s) = \frac{1}{4}(s^2-1)^2$ that has two minima at $s = \pm 1$.  Various unconditionally stable numerical schemes have been proposed to approximate similar gradient-flow types of phase field equations, among which we mention the Convex-concave splitting (CCS) approach \cite{EllStu,Eyre}; the Stabilized linearly implicit (SLI) approach \cite{ShenYang}; the Invariant energy quadratization (IEQ) approach \cite{IEQ:diblock} \tc{(which is based on an earlier Lagrange multiplier approach \cite{IEQ})}; and the Scalar auxiliary variable approach \cite{SAV1,SAV2} which is based on the simple idea of introducing a scalar variable 
\begin{align}\label{scalarvar}
q(t) = \Big ( \int_\Omega \frac{1}{\eps} F(\vp(t)) \dx + C_0 \Big )^{1/2} =: Q(\vp(t)) > 0,
\end{align}
for potentials $F$ such that $\int_\Omega \frac{1}{\eps} F(\vp) \dx$ is bounded strictly from below by $-C_0$ with constant $C_0 > 0$, so that an ordinary differential equation can be obtained:
\[
q_t = \frac{1}{2 \eps Q(\vp)} \int_\Omega F'(\vp) \pd_t \vp \dx, \quad q(0) = Q(\vp_0).
\]
The equivalent SAV formulation of \eqref{CH} in strong form reads as 
\begin{subequations}\label{CH:alt}
\begin{alignat}{3}
\pd_t \vp& = \div (m(\vp) \nabla \mu) + f  \quad && \text{ in } \Omega \times (0,T), \\
\mu & = - \eps \Lap \vp + \frac{q(t)}{\eps Q(\vp)} F'(\vp) - g \quad && \text{ in } \Omega \times (0,T), \label{CH:alt:2}  \\
q_t & = \frac{1}{2 \eps Q(\vp)} \int_\Omega F'(\vp) \pd_t \vp \dx \quad  && \text{ in } (0,T), \label{CH:alt:3}
\end{alignat}
\end{subequations}
which admits an analogous energy identity to \eqref{CH:energy} but with $E_\eps$ replaced by a modified energy $G_\eps = \frac{\eps}{2} \| \nabla \vp \|^2 + |q|^2$. A first order time discretization would take $F'(\vp)/(\eps Q(\vp))$ explicitly and $q$ implicitly when approximating \eqref{CH:alt:2}, yielding a numerical scheme that is linear in the next iteration for $(\vp, \mu, q)$. In comparison to the CCS approach which results in nonlinear discrete systems, or the SLI approach which requires a quadratic truncation of $F$ and large stabilization parameters, the IEQ and SAV approaches allow for unconditionally stable linear schemes when applied to gradient flows \cite{SAV2,Shen:IEQ} without the need to modify the potential $F$. Between the latter two, the SAV approach is viewed as an enhancement of the IEQ approach, whose main idea of introducing an auxiliary function $u = \sqrt{F(\vp)}$ and replacing $(q(t)F'(\vp))/(\eps Q(\vp))$ with $(F'(\vp) u)/(\eps \sqrt{F(\vp)})$ in \eqref{CH:alt:2}, and replacing \eqref{CH:alt:3} with a PDE $\pd_t u = (F'(\vp) \pd_t \vp)/(2\sqrt{F(\vp)})$ demands $F$ to be a strictly positive potential function.

Since its introduction, many variants and improvements of the SAV approach have been developed, among which we mention the Lagrange multiplier approach \cite{Cheng} that dissipates the original energy as opposed to a modified energy; the relaxed SAV approach \cite{RSAV} penalizing large differences between $q^n$ and $Q(\vp^n)$ at the discrete level; SAV variants \cite{Hou,Liu} removing the lower bound assumption on $\int_\Omega F(\vp) \dx$; generalized SAV (GSAV) \cite{GSAV,Roadmap} and GSAV with relaxation (R-GSAV) \cite{Zhang} approaches applicable to general dissipative systems of the form
\begin{align}\label{diss}
\pd_t \phi + \mathcal{A}(\phi) + g(\phi) = 0 \quad \text{ such that } \quad \frac{d}{dt} E(\phi) = - \mathcal{K}(\phi),
\end{align}
with positive operator $\mathcal{A}$, semi-linear/quasi-linear operator $g$, free energy functional $E$ and dissipative functional $\mathcal{K}$ such that $\mathcal{K}(\phi) > 0$ for all $\phi$.

Concerning the numerical analysis, assuming sufficient regularity on the solution to the underlying PDE, which is either a gradient flow of some energy functional or a dissipative system of the form \eqref{diss}, first error estimates for the standard SAV approach were established in \cite{SAVconv:grad,SAVconv} for the Allen--Cahn and Cahn--Hilliard equations; in \cite{Li:CHS} for a Cahn--Hilliard--Stokes system; in \cite{Zheng:CHHS} for a Cahn--Hilliard-Hele-Shaw system. For the GSAV and R-GSAV approaches, error analysis with backwards differentiation formula (BDF) time stepping can be found in \cite{GSAV:err,Zhang}.  On the other hand, convergence of discrete solutions without the assumed smoothness of the continuous solution can  be found in \cite{SAVconv} for the $L^2$- and $H^{-1}$-gradient flow of $E_\eps$, and in \cite{Metzger} for the bulk-surface Cahn--Hilliard system of \cite{KLLM}.

The motivation of the present work stems from the observations that (i) Cahn--Hilliard systems with mass sources are becoming ubiquitous in various areas of applied sciences; (ii) for a generic source term $f = f(x)$ or $f = f(x,\vp)$, the system \eqref{CH} is no longer a gradient flow of $E_\eps$, and thus new ideas are needed to establish the stability of SAV discrete solutions that was automatically guaranteed in previous works. In our analysis we notice that if $f \in L^2(\Omega)$, or $|f(\vp)| \leq C(1 + |\vp|)$ having linear growth, then stability estimates require the potential $F$ to have at most quadratic growth (see Section~\ref{sec:discF}). This restriction has also been observed in \cite{GarLam,GaTr} for a Cahn--Hilliard tumor model \eqref{Tumor} with non-constant mobility function $m(\vp)$. This is similar to the SLI approach where $F$ needs to be modified so that $\| F'' \|_{L^\infty} < \infty$.  However, as we recall one of the advantages of the SAV approach over the SLI approach is that $F$ needs no ad-hoc modification, the basis of our contribution is to provide stability estimates for SAV discrete solutions to \eqref{CH} whilst keeping $F$ unmodified, which can be attained when $f$ is a bounded function. This boundedness assumption is not restrictive, as seen from several applications of \eqref{CH} outlined in Section~\ref{sec:appl}, and it is possible to extend our stability analysis to Cahn--Hilliard systems with non-bounded sources $f = f(\vp)$ provided the time evolution of the mean value of $\vp$ can be controlled.

Our main contributions are as follows: (i) we formulate a linear numerical scheme for \eqref{CH} based on the RSAV approach of \cite{RSAV}, and demonstrate first results on the stability of discrete solutions for bounded $f$ and potential $F$ that admits quartic polynomial growth; (ii) we prove first convergence results for the time discrete solutions to a weak solution of \eqref{CH}.  To the best of our knowledge, SAV schemes for \eqref{CH} has not received much attention in the literature, which we attribute to the lack of an obvious Lyapunov functional, or a dissipative structure akin to \eqref{diss} for the equation. Furthermore, we are not aware of any results concerning the convergence of RSAV schemes of \cite{RSAV} (besides the R-GSAV variant in \cite{Zhang} for dissipative systems). As a consequence, notable modified Cahn--Hilliard models in the literature for diblock co-polymer phase separation, tumor growth, image inpainting and segmentation can now be approach numerically with the SAV framework, thereby extending its applicability to systems that need not exhibit a dissipative structure.

\textbf{Plan of the paper}: We present the RSAV time discretization for \eqref{CH} in Section~\ref{sec:SAV}, and establish stability estimates.  Section~\ref{sec:conv} is devoted to the convergence analysis of the time discrete solutions, and we discuss supporting numerical results in Section~\ref{sec:num}. In Section~\ref{sec:appl} we outline SAV schemes, their stability and numerical simulations for Cahn--Hilliard system with mass source in biological science and image processing.

\textbf{Notation}: For any $p \in [1,\infty]$ and $k >0$, the standard Lebesgue and Sobolev spaces over $\Omega$ are denoted by $L^p(\Omega)$ and $W^{k,p}(\Omega)$ with the corresponding norms $\no{\cdot}_{L^p}$ and $\no{\cdot}_{W^{k,p}}$.  In the special case $p = 2$, these become Hilbert spaces and we employ the notation $H^k := H^k(\Omega) = W^{k,2}(\Omega)$ with the corresponding norm $\no{\cdot}_{H^k}$. We denote the topological dual of $H^1(\Omega)$ by $(H^1(\Omega))^*$ and the corresponding duality pairing by $\inn{\cdot}{\cdot}$.  We use $\no{\cdot}$ and $(\cdot,\cdot)$ for the norm and inner product of $L^2(\Omega)$.

\section{RSAV scheme}\label{sec:SAV}
\subsection{First order time discretization}
Dividing the time interval $[0,T]$ into a uniform partition of subintervals $[t^{n-1}, t^{n}]$ with $\tau = t^{n} - t^{n-1}$, for $n = 1, \dots, N_\tau$, where $\tau N_\tau \leq T$. Denoting by $f^k, g^k$ as approximations of $f , g$ evaluated at $t = t^k$ for $k = 1, \dots, N_\tau$, then a first order time discretization of \eqref{CH:alt} based on the SAV approach with relaxation proposed in \cite{RSAV} reads as follows: Let $\tau > 0$ denote a fixed time step and given $(\vp^{n-1}, q^{n-1}, f^{n-1} ,g^{n-1})$, find $(\vp^{n}, \mu^{n}, r^{n})$ satisfying (in a suitable sense)
\begin{subequations}\label{CH:dis}
\begin{alignat}{2}
\vp^{n} - \vp^{n-1} & = \tau\, \div (m(\vp^{n-1}) \nabla \mu^{n}) + \tau f^{n-1}, \label{dis:1} \\
\mu^{n} & = - \eps \Lap \vp^{n} + \frac{r^{n}}{\eps Q(\vp^{n-1})} F'(\vp^{n-1}) - g^{n-1}, \label{dis:2} \\
r^{n} - q^{n-1} & = \frac{1}{2 \eps Q(\vp^{n-1})} (F'(\vp^{n-1}), \vp^{n} - \vp^{n-1}), \label{dis:3}
\end{alignat}
\end{subequations}
initialized with $\vp^0 = \vp_0$ and $q^0 = Q(\vp_0)$. Then, we define 
\begin{align}\label{dis:4}
q^n :=  \zeta_n^\tau r^n + (1-\zeta_n^\tau) Q(\vp^{n})
\end{align}
for some $\zeta_n^\tau \in \V_{n}^\tau$, where given $\tau, \mu^n, r^n$ and manually assigned fixed constants $\eta \in (0,1)$ and $M > 0$, the set $\V_{n}^\tau$ is defined as the set of constants $\zeta \in [0,1]$ such that 
\begin{equation}\label{dis:5}
\begin{aligned}
0 \geq R_n^\tau(\zeta) &:= |\zeta r^n + (1-\zeta) Q(\vp^{n})|^2 + \frac{1}{2} |\zeta r^n + (1-\zeta) Q(\vp^{n}) - q^{n-1}|^2  \\
 & \quad - \tau \eta m_0 \| \nabla \mu^n \|^2 - |r^n|^2 - \frac{1}{2} |r^n - q^{n-1}|^2 - \tau M,
\end{aligned}
\end{equation}
\tc{where $m_0$ is the positive lower bound on the mobility function, see \eqref{ass:const:m} below}. Our definition of $\V_n^\tau$ differs slightly from \cite{RSAV}. The standard SAV method is obtained when $\zeta = 1$. Note that $R_n(1) \leq - \tau M < 0$ and so $1 \in \V_n^\tau$, i.e., $\V_n^\tau$ is non-empty.  In fact, $R_n^\tau(\zeta) = a_n^\tau \zeta^2 + b_n^\tau \zeta + c_n^\tau$ with coefficients
\begin{align*}
a_n^\tau & = \frac{3}{2}(r^n - Q(\vp^n))^2, \quad b_n^\tau = (r^n - Q(\vp^n))(Q(\vp^n) - q^{n-1}), \\
c_n^\tau & = |Q(\vp^n)|^2 - |r^n|^2 + \frac{1}{2}[|Q(\vp^n) - q^{n-1}|^2 - |r^n - q^{n-1}|^2] - \tau \eta m_0 \| \nabla \mu^n \|^2 - \tau M.
\end{align*}
Since $R_n^\tau(1) < 0$, by continuity we can find $w_n^\tau \in [0,1)$ such that $R_n^\tau(\zeta) \leq 0$ for all $\zeta \in [w_n^\tau,1]$, and so we can pick any $\zeta_n^\tau \in (w_n^\tau,1]$ for the update rule \eqref{dis:4}.

\begin{remark}\label{rmk:ideal}
It would be ideal to pick $\zeta_n^\tau = 0$ so that $q^n = Q(\vp^n)$, \tc{which we term as the idealized update approach when discussing convergence in Section \ref{sec:conv} below}.  However, for the Cahn--Hilliard system \eqref{CH} with generic mass source $f$, we cannot guarantee if $R_n^\tau(0) \leq 0$ is satisfied \tc{for all $n = 1, \dots, N_\tau$} as there is no available comparison between $|Q(\vp^n)|^2$ and $|r^n|^2$ using the discrete energy inequality \eqref{stab:1} of \eqref{CH:dis}. For the numerical simulations performed in \cite{RSAV}, the following optimal value for $\zeta_n^\tau$ is suggested:
\begin{align}\label{RSAV:opt}
\zeta_n^\tau = \max \Big (0, \frac{-b_n^\tau - \sqrt{(b_n^\tau)^2 - 4a_n^\tau c_n^\tau}}{2a_n^\tau} \Big ),
\end{align}
where the values of $a_n^\tau$, $b_n^\tau$ and $c_n^\tau$ at each iteration are defined above. In our numerical tests reported in Section~\ref{sec:num}, we found that \tc{the optimal value} $\zeta_n^\tau = 0$ \tc{can be achieved} with moderate values of $M$ and $\eta$.
\end{remark}

\subsection{Stability estimates}
We make the following assumptions:
\begin{enumerate}[label=$(\mathrm{A \arabic*})$, ref = $\mathrm{A \arabic*}$]
\item \label{ass:dom} $\Omega \subset \R^{d}$, $d \in \{2,3\}$ is a bounded domain with convex or $C^{1,1}$-boundary $\pd \Omega$.
\item \label{ass:const:m} The constant $\eps > 0$ is fixed, and the mobility function $m \in C^0(\R)$ satisfies
\[
m_0 \leq m(s) \leq m_1 \quad \forall s \in \R
\]
for some constants $m_1 \geq m_0 > 0$.
\item \label{ass:ini} The initial condition satisfies $\phi_0 \in H^1(\Omega)$.
\item \label{ass:g} The function $g: \Omega \times (0,T) \to \R$ satisfies $g \in L^\infty(0,T;L^2(\Omega))$.
\item \label{ass:Ff} The potential $F$ is a non-negative function with $F \in C^1(\R)$, and there exist constants $c_0 > 0$, $c_1 \in \R$, $c_2 > 0$ such that 
\[
c_0 |s|^3 - c_1 \leq |F'(s)| \leq c_2 (1 + |s|^3) \quad \forall s \in \R,
\]
and the source term $f: \Omega \times (0,T) \to \R$ satisfies $f \in L^\infty(0,T;L^\infty(\Omega))$.
\end{enumerate}
In light of \eqref{ass:g} and \eqref{ass:Ff} for the source functions, we introduce the Steklov averages $f^\tau$ and $g^\tau$ defined as
\[
f^\tau(t,x) = \frac{1}{\tau} \int_t^{t+\tau} \tilde{f}(y, x) \dy, \quad g^\tau(t,x) = \frac{1}{\tau} \int_t^{t+\tau} \tilde{g}(y,x) \dy
\]
for $t \in [0,T-\tau]$, where $\tilde{f}$ and $\tilde{g}$ are zero extensions of $f$ and $g$ on $t \in \R \setminus [0,T]$. Then, well-known properties of Steklov averages yield $f^\tau \in C^0([0,T];L^\infty(\Omega))$, $g^\tau \in C^0([0,T];L^2(\Omega))$, $\| f^\tau \|_{L^\infty(0,T;L^\infty)} \leq \| f \|_{L^\infty(0,T;L^\infty)}$, $\| g^\tau \|_{L^\infty(0,T;L^2)} \leq \| g \|_{L^\infty(0,T;L^2)}$, with $f^\tau \to f$ in $L^r(0,T;L^\infty(\Omega))$ and $g^\tau \to g$ in $L^r(0,T;L^2(\Omega))$ as $\tau \to 0$ for any $1 \leq r < \infty$. Hence, we can choose $f^{n-1}(x) := f^\tau(t^{n-1},x)$ and $g^{n-1}(x) := g^\tau(t^{n-1},x)$ in \eqref{CH:dis}.

\begin{remark}[Unique solvability]
The existence of time discrete solutions to \eqref{CH:dis} can be inferred almost analogously as in Section~\ref{sec:FE} by replacing the corresponding matrices with appropriate differential operators, and so we omit the details. Concerning uniqueness of solutions, as \eqref{CH:dis} is linear in $(\vp^{n}, \mu^{n}, r^{n})$, the differences denoted by $(\vp, \mu, r)$ between any two solution triplets satisfy formally in the strong sense
\[
\vp = \tau \div(m(\vp^{n-1}) \nabla \mu), \quad \mu = -\eps \Lap \vp + \frac{r F'(\vp^{n-1})}{\eps Q(\vp^{n-1})}, \quad r = \frac{(F'(\vp^{n-1}), \vp)}{2 \eps Q(\vp^{n-1})}.
\]
Testing with $\mu$, $\vp$ and $2r$, respectively, and upon summing we get
\[
\tau \int_\Omega m(\vp^n) |\nabla \mu|^2 \dx + \eps \| \nabla \vp \|^2 +2 |r|^2 = 0.
\]
Hence $r = 0$, while $\vp$ and $\mu$ are constants. Since $\nabla \mu = \bm{0}$, we infer from the first equation $\vp = 0$, and consequently $\mu = 0$. Then, $q^n$ is uniquely determined from \eqref{dis:4}.
\end{remark}

\begin{remark}
\eqref{ass:Ff} essentially requires that the potential $F$ has quartic polynomial growth at infinity, which is fulfilled by the classical example $F(s) = \frac{1}{4}(s^2-1)^2$.  We show in Remark \ref{rem:sharp} that this seems to be sharp for our main stability result.
\end{remark}

\begin{remark}
It is also possible to consider $g \in L^2(0,T;H^1(\Omega))$ as oppose to \eqref{ass:g}, since we can define a new variable $\hat \mu := \mu + g$ and shift $g$ from \eqref{CH:2} into the diffusion term of \eqref{CH:1}, see Section~\ref{sec:Tumor} for an example.
\end{remark}

Our main result is the following stability estimate for the time discrete solutions.
\begin{theorem}[Stability]\label{thm:stab}
Suppose, for any $\tau \in (0,1)$, the time discrete system \eqref{CH:dis} is solvable.  Then, under \eqref{ass:dom}-\eqref{ass:Ff}, there exists $\tau_* \in (0,1)$ depending only on model parameters, such that for all $\tau \in (0,\tau_*)$, the following estimate holds for the corresponding discrete solutions $\{\vp^k, \mu^k, r^k, q^k \}_{k=1}^{N_\tau}$, where $\tau N_\tau \leq T$, with a positive constant $C$ independent of $\tau$ and $N_\tau$:
\begin{equation}\label{stab}
\begin{aligned}
& \| \vp^k \|_{H^1}^2 + |q^k|^2 + |r^k|^2 + \sum_{n=1}^{k} \Big ( \| \vp^{n} - \vp^{n-1} \|_{H^1}^2 + |r^{n} - q^{n-1} |^2 + |q^{n} - q^{n-1} |^2 \Big ) \\
& \quad + \tau \sum_{n=1}^{k}  \Big (\| \mu^{n} \|_{H^1}^2 + \| \vp^{n} \|_{H^2}^4 \Big) + \sum_{n=1}^{k} \frac{1}{\tau} \|\vp^{n} - \vp^{n-1}\|_{(H^1)^*}^2 \leq C,
\end{aligned}
\end{equation}
and for any $h \in \{2, \dots, N_\tau\}$,
\begin{align}\label{timediff}
\tau \sum_{n=0}^{N_\tau - 1} \| \vp^{n+1} - \vp^n \|^2 \leq C\tau^{3/2}, \quad \tau \sum_{n=0}^{N_\tau - h} \| \vp^{n+h} - \vp^n \|^2 \leq C h \tau.
\end{align}
\end{theorem}

\begin{proof}
Let $\{\vp^k, \mu^k, r^k, q^k \}_{k=1}^{N_\tau}$ denote the discrete solutions to \eqref{CH:dis}-\eqref{dis:4} originating from the initial conditions $\vp^0 := \vp_0$ and $q^0 := Q(\vp_0)$. For $k = n$, we test \eqref{dis:1} with $\mu^{n}$ and also with $\vp^{n}$, \eqref{dis:2} with $\vp^{n} - \vp^{n-1}$ and \eqref{dis:3} with $2r^{n}$. Then, employing the identity $(a-b)a = \tfrac{1}{2}(a^2 - b^2 + (a-b)^2)$ we first obtain upon summing
\begin{align*}
& \frac{1}{2} \Big ( \| \vp^n \|^2 + \eps \| \nabla \vp^n \|^2 + 2 |r^n|^2 \Big ) - \frac{1}{2}\Big ( \| \vp^{n-1} \|^2 + \eps \| \nabla \vp^{n-1} \|^2 + 2|q^{n-1}|^2 \Big ) \\
& \qquad + \frac{1}{2} \Big ( \| \vp^n - \vp^{n-1} \|^2 + \eps \| \nabla (\vp^{n} - \vp^{n-1}) \|^2 \Big ) + |r^{n} - q^{n-1}|^2 + \int_\Omega \tau m(\vp^{n-1}) |\nabla \mu^n|^2 \dx \\
& \quad = \tau (f^{n-1}, \mu^n) + \tau(f^{n-1}, \vp^n) + (g^{n-1}, \vp^n - \vp^{n-1}) - \tau(m(\vp^{n-1}) \nabla \mu^n, \nabla \vp^n).
\end{align*}
Combining with the following inequality derived from \eqref{dis:4} and \eqref{dis:5}
\[
|q^n|^2 + \frac{1}{2}|q^n - q^{n-1}|^2 \leq \tau M + \tau \eta m_0 \| \nabla \mu^n\|^2 + |r^n|^2 + \frac{1}{2}|r^n - q^{n-1}|^2,
\]
we obtain the inequality 
\begin{equation}\label{stab:1}
\begin{aligned}
& G^{n} - G^{n-1} + \frac{1}{2} \| \vp^{n} - \vp^{n-1} \|^2 + \frac{\eps}{2}\|\nabla (\vp^{n} - \vp^{n-1})\|^2 + \frac{1}{2} |q^n - q^{n-1}|^2 \\
& \qquad + \frac{1}{2} |r^{n} - q^{n-1}|^2 + \int_\Omega (1-\eta) \tau m_0  |\nabla \mu^{n}|^2 \dx \\
& \quad \leq  \tau (f^{n-1}, \mu^{n}) + \tau M + \tau (f^{n-1}, \vp^{n}) + (g^{n-1}, \vp^{n} - \vp^{n-1}) \\
& \qquad- \tau (m(\vp^{n-1}) \nabla \mu^{n}, \nabla \vp^{n}),
\end{aligned}
\end{equation}
where
\begin{align}\label{def:G}
G^k := \frac{1}{2} \| \vp^k \|^2 + \frac{\eps}{2} \| \nabla \vp^k \|^2 + |q^k|^2
\end{align}
plays the role of a discrete energy functional. The latter three terms on the right-hand side of \eqref{stab:1} can be handled in a standard way, where for the remainder of the proof, the symbol $C$ denotes a generic constant independent of $\tau$ and $n \in \{1, \dots, N_\tau\}$ \tc{whose value may change line from line and within the same line}:
\begin{equation}\label{stab:1:rhs}
\begin{aligned}
 & \tau (f^{n-1}, \vp^{n}) + (g^{n-1}, \vp^{n} - \vp^{n-1}) - \tau (m(\vp^{n-1}) \nabla \mu^{n}, \nabla \vp^{n})\\
& \quad \leq C \tau  + C \tau \| \vp^{n} \|_{H^1}^2 + C \| g^{n-1} \|^2 + \frac{1}{4} \| \vp^{n} - \vp^{n-1} \|^2 + \frac{\tau(1-\eta) m_0}{4} \|\nabla \mu^{n} \|^2.
\end{aligned}
\end{equation}
For the first term we use of the boundedness of $f$ and the Poincar\'e inequality to infer
\begin{equation}\label{mu:trick}
\begin{aligned}
|(f^{n-1}, \mu^{n})| & \leq C \| \mu^{n} \|_{L^1} \leq C \| \nabla \mu^{n} \|_{L^2} + C | (\mu^{n},1)| \\
& \leq C + \frac{(1-\eta) m_0}{4} \| \nabla \mu^{n} \|^2 + C|(\mu^{n},1)|.
\end{aligned}
\end{equation}
To close the estimate, it remains to estimate the mean value of $\mu^{n}$, where from \eqref{dis:2},
\begin{align}\label{mean:mu}
C|(\mu^{n},1)| \leq C\frac{|r^{n}|}{\eps Q(\vp^{n-1})} |(F'(\vp^{n-1}),1)| + C|(g^{n-1},1)|.
\end{align}
By virtue of \eqref{ass:Ff}, we have $|F'(s)| \leq C F(s)^{\frac{3}{4}} + C$ and so
\begin{align}\label{F:F'}
|(F'(\vp^{n-1}),1)| \leq C (F(\vp^{n-1}),1)^{\frac{3}{4}} + C.
\end{align}
Combining with the facts $\eps Q(\vp^{n-1})^2 = (F(\vp^{n-1}),1) + \eps C_0$ and $Q(\vp^{n-1}) \geq \sqrt{C_0}$, and the Sobolev embedding $H^1(\Omega) \subset L^4(\Omega)$, we deduce the following key estimate:
\begin{equation}\label{mean:mu:1}
\begin{aligned}
C \frac{|r^{n}| |(F'(\vp^{n-1}),1)|}{\eps Q(\vp^{n-1})} & \leq C \frac{|r^{n}| ( (F(\vp^{n-1}), 1)^{\frac{3}{4}} + 1 )}{Q(\vp^{n-1})}  \leq C |r^{n}| \, (Q(\vp^{n-1})^{\frac{1}{2}} + 1) \\
&  \leq \frac{1}{8} |r^{n}|^2 + C \big ( 1 + Q(\vp^{n-1}) \big )  \leq \frac{1}{8} |r^{n}|^2 + C \big ( 1 + (F(\vp^{n-1}), 1)^{1/2} \big ) \\
&   \leq  \frac{1}{8} |r^{n}|^2 + C \big ( 1 +  \| \vp^{n-1} \|_{L^4}^2 \big ) \leq \frac{1}{8} |r^{n}|^2 + C \big ( 1 +  \| \vp^{n-1} \|_{H^1}^2 \big ).
\end{aligned}
\end{equation}
Substituting into \eqref{mean:mu} yields the following estimate for the mean value
\begin{equation}\label{mean:mu:est}
\begin{aligned}
C |(\mu^{n},1)| & \leq C + C \| g^{n-1} \|^2 + \frac{1}{8} |r^{n}|^2 + C \| \vp^{n-1} \|_{H^1}^2 \\
& \leq C + C\| g^{n-1} \|^2 + \frac{1}{4} |r^n - q^{n-1}|^2 + C |q^{n-1}|^2 + C \| \vp^{n-1} \|_{H^1}^2,
\end{aligned}
\end{equation}
and consequently
\begin{equation}\label{stab:1:rhs:1}
\begin{aligned}
\tau |(f^{n-1},\mu^{n})| & \leq C \tau + \frac{(1-\eta)\tau m_0}{4} \| \nabla \mu^{n}\|^2 + \frac{\tau}{4} |r^n - q^{n-1}|^2 \\
& \quad + C \tau \| g^{n-1} \|^2 + C \tau |q^{n-1}|^2 + C \tau \| \vp^{n-1} \|_{H^1}^2.
\end{aligned}
\end{equation}
Using \eqref{stab:1:rhs} and \eqref{stab:1:rhs:1} in \eqref{stab:1}, we arrive at the following:
\begin{align*}
& G^{n} +  \frac{\min(1,2\eps)}{4} \| \vp^{n} - \vp^{n-1} \|_{H^1}^2 + \frac{1-\tau}{4}  |r^{n} - q^{n-1}|^2 + \frac{1}{2} |q^n - q^{n-1}|^2 + \frac{m_0 (1-\eta) \tau}{2} \| \nabla \mu^{n} \|^2 \\
& \quad \leq G^{n-1} + C \tau ( 1 + \| g^{n-1} \|^2 + G^{n-1} + G^n) + C \| g^{n-1} \|^2.
\end{align*}
For $\tau < 1$ and arbitrary $k \in \{1, \dots, N_\tau\}$, taking the sum of the above inequality from $n = 1$ to $n = k$ yields
\begin{align*}
& (1-C \tau) G^k +  \frac{\min(1,2\eps)}{4}\sum_{n=1}^{k} \| \vp^{n} - \vp^{n-1} \|_{H^1}^2 + \frac{(1-\tau)}{4} \sum_{n=1}^{k}  |r^{n} - q^{n-1}|^2 \\
& \qquad + \frac{1}{2}\sum_{n=1}^{k}  |q^{n} - q^{n-1}|^2+ \frac{\tau(1-\eta) m_0}{2} \sum_{n=1}^{k} \| \nabla \mu^{n} \|^2 \\
& \quad  \leq G^0 + C \sum_{n=1}^{k} \| g^{n-1} \|^2 + C \tau \sum_{n=1}^{k} (1 + \| g^{n-1} \|^2 + G^{n-1} ).
\end{align*}
Let $\tau_*$ be a fixed constant so that the prefactor $1-C\tau_*$ is positive,  Then, for all $\tau \in (0,\tau_*)$, by the discrete Gronwall inequality it holds that for all $k = 1, \dots, N_\tau$,
\begin{equation}\label{stab:2}
\begin{aligned}
& \| \vp^k \|_{H^1}^2 + |q^{k}|^2 + \sum_{n=1}^{k} \| \vp^{n} - \vp^{n-1} \|_{H^1}^2 + \sum_{n=1}^{k}  |r^{n} - q^{n-1}|^2 \\
& \quad +\sum_{n=1}^{k}  |q^{n} - q^{n-1}|^2 + \tau \sum_{n=1}^{k} \| \nabla \mu^{n} \|^2 \leq C.
\end{aligned}
\end{equation}
In turn, we find that for all $k = 1, \dots, N_\tau$,
\begin{align}\label{rn:est}
|r^k|^2 \leq 2|r^k - q^{k-1}|^2 + 2 |q^{k-1}|^2 \leq C.
\end{align}
Returning to the mean value estimate \eqref{mean:mu:est} and by the Poincar\'e inequality, we find 
\begin{align}\label{stab:3}
\tau \sum_{n=1}^{k} |(\mu^{n},1)| \leq C \quad \text{ and } \quad \tau \sum_{n=1}^{k} \| \mu^{n} \|^2 \leq C.
\end{align}
Next, we test \eqref{dis:1} with an arbitrary $\zeta \in H^1(\Omega)$, leading to 
\[
\frac{1}{\tau}|(\vp^{n} - \vp^{n-1}, \zeta)| \leq m_1 \| \nabla \mu^{n} \| \| \nabla \zeta \| + \| f^n \| \| \zeta \| \leq C \Big ( 1 + \|\nabla \mu^{n} \| \Big ) \| \zeta \|_{H^1}.
\]
This shows that $\| \tfrac{1}{\tau}(\vp^{n} - \vp^{n-1}) \|_{(H^1)^*} \leq  C ( 1 + \|\nabla \mu^{n} \| )$ and hence
\[
\sum_{n=1}^{k}\frac{1}{\tau} \| \vp^{n} - \vp^{n-1} \|_{(H^1)^*}^2 \leq C.
\]
To show \eqref{timediff}, we test \eqref{dis:1} with $\vp^{m+h} - \vp^{m}$ where $m = 0, \dots, N_\tau - h$ and $h = 1, \dots, N_\tau$ to obtain
\[
0 = (\vp^n - \vp^{n-1} - \tau f^{n-1}, \vp^{m+h} - \vp^m) + \tau (m(\vp^{n-1}) \nabla \mu^n, \nabla (\vp^{m+h} - \vp^m)).
\]
Summing from $n = m+1,\dots, m+h$ yields
\begin{align*}
\| \vp^{m+h} - \vp^m \|^2 & \leq C\tau  \sum_{n=m+1}^{m+h} \Big ( \| f^{n-1} \| + \| \nabla \mu^n \| \Big ) \| \vp^{m+h} - \vp^m \|_{H^1} \\
& \leq C \tau \sum_{k=1}^h \Big ( 1 + \| \nabla \mu^{m+k} \| \Big ) \| \vp^{m+h} - \vp^m \|_{H^1}.
\end{align*}
Multiplying both sides by $\tau$, summing from $m = 0, \dots, N_\tau - h$ and applying H\"older's inequality leads to 
\begin{align*}
& \tau \sum_{m=0}^{N_\tau - h} \| \vp^{m+h} - \vp^m \|^2 \leq C \tau^2 \sum_{k=1}^h \sum_{m=0}^{N_\tau - h}  \Big ( 1 + \| \nabla \mu^{m+k} \| \Big ) \| \vp^{m+h} - \vp^m \|_{H^1} \\
&\quad \leq C \tau \sum_{k=1}^h \Big ( \tau \sum_{m=0}^{N_\tau - h} \| \vp^{m+h} - \vp^m \|_{H^1}^2 \Big )^{1/2} \Big (1 + \Big ( \tau \sum_{m=0}^{N_\tau - h} \| \nabla \mu^{m+k} \|^2 \Big )^{1/2} \Big ) \\
& \quad \leq \tc{C \tau \sum_{k=1}^h \Big (N_\tau \tau \max_{n = 1, \dots, N_\tau} \| \vp^n \|_{H^1}^2 \Big )^{1/2} \Big ( 1 + \Big ( \tau \sum_{n=0}^{N_\tau} \| \nabla \mu^n \|^2 \Big )^{1/2} \Big )} \leq C h \tau.
\end{align*}
When $h = 1$, a higher exponent can be inferred by using \eqref{stab:2}:
\[
\tau \sum_{m=0}^{N_\tau -1} \| \vp^{m+1} - \vp^m \|^2 \leq C \tau \Big ( \tau \sum_{m=0}^{N_\tau - 1} \| \vp^{m+1} - \vp^m \|_{H^1}^2 \Big )^{1/2} \leq C \tau^{3/2}.
\]
Furthermore, we test \eqref{dis:2} with $- \Lap \vp^{n}$, which yields
\begin{equation}\label{H2:est}
\begin{aligned}
\frac{\eps}{2} \| \Lap \vp^{n} \|^2 & \leq C \| g^{n-1} \|^2 + C \frac{|r^{n}|^2}{Q(\vp^{n-1})^2} \| F'(\vp^{n-1}) \|^2 \\
& \quad + C \| \nabla \mu^{n} \| \| \nabla \vp^{n} \| \leq C + C\| \nabla \mu^{n} \|,
\end{aligned}
\end{equation}
after applying \eqref{stab:2}, \eqref{rn:est}, the lower bound $Q(\vp^{n-1}) \geq \sqrt{C_0}$, and \eqref{ass:Ff}. Invoking elliptic regularity (e.g.~\cite[Thm.~2.4.2.7]{Gris}), we find that 
\begin{align}\label{stab:4}
\tau \sum_{n=1}^{k} \| \vp^{n} \|_{H^2}^4 \leq C \tau \sum_{n=1}^{k} \big ( \| \Lap \vp^{n} \|^4 + \| \vp^{n} \|^4 \big ) \leq C.
\end{align}
This completes the proof.
\end{proof}

\begin{remark}[Sharpness of \eqref{ass:Ff}]\label{rem:sharp}
If $F$ satisfies the growth assumptions 
\[
c |s|^p - C \leq F(s) \leq C(1 + |s|^p), \quad  |F'(s)| \leq C(1+|s|^{p-1}),
\]
for some $p \in [2,\infty)$ if $d = 2$, or $p \in [2,6)$ if $d = 3$, via a similar calculation as \eqref{mean:mu:1}, we find that 
\begin{align*}
\frac{|r^{n}|}{\eps Q(\vp^{n-1})} |(F'(\vp^{n-1}),1)| & \leq C \frac{|r^{n}|}{Q(\vp^{n-1})} (F(\vp^{n-1}),1)^{\frac{p-1}{p}} \\
& \quad \leq \frac{1}{8} |r^{n}|^2 + C (F(\vp^{n-1}),1)^{\frac{p-2}{p}} \leq \frac{1}{8} |r^{n}|^2 + C \| \vp^{n-1} \|_{L^p}^{p-2},
\end{align*}
where we have omitted lower order terms. Let $\alpha = \frac{p-2}{p}$ if $d = 2$, or $\alpha = \frac{3(p-2)}{2p}$ if $d = 3$, so that $\alpha \in (0,1)$, and by invoking the Gagliardo--Nirenberg inequality,
\begin{align*}
\| \vp^{n-1} \|_{L^p}^{p-2} & \leq C \| \vp^{n-1} \|_{L^2}^{(1-\alpha)(p-2)} \| \vp^{n-1} \|_{H^1}^{\alpha(p-2)}  \leq C \| \vp^{n-1} \|_{L^2}^{\frac{s(1-\alpha)(p-2)}{s-1}} + C \| \vp^{n-1} \|_{H^1}^{\alpha(p-2)s},
\end{align*}
for some $s > 1$. Hence, to control $\| \vp^{n-1} \|_{L^p}^{p-2}$ with $\| \vp^{n-1} \|_{H^1}^2$, we demand 
\[
\alpha(p-2) s = \frac{s(1-\alpha)(p-2)}{s-1} = 2,
\]
and a short calculation shows that $p = 4$. Hence, to use the idea of \eqref{mean:mu:1} to proceed, the quartic growth assumption \eqref{ass:Ff} on $F$ seems to be sharp.
\end{remark}

\subsection{Discussion on the polynomial growth of $F$}\label{sec:discF}
The boundedness assumption \eqref{ass:Ff} on $f$ is essential for the above analysis to go through if one chooses $F$ to be a polynomial potential with super-quadratic growth.  In particular, if $f$ is assumed to belong to $L^2(0,T;L^2(\Omega))$, or treating $f = \hat{f}(\vp)$ where $|\hat{f}(s)| \leq C(1 + |s|)$ has linear growth, then the analysis restricts $F$ to have at most quadratic growth. Indeed, we revisit \eqref{mu:trick} where upon employing the Poincar\'e inequality:
\[
|(f^{n-1}, \mu^{n})| \leq C \| f^{n-1} \| \| \mu^{n} \| \leq C \| f^{n-1} \| \big (\| \nabla \mu^{n} \| + |(\mu^{n},1)| \big ).
\]
Control of this term requires an a priori estimate on the square of the mean value in terms of $|r^{n}|^2$ and $\| \vp^{n} \|_{H^1}^2$. From \eqref{mean:mu} we see
\[
|(\mu^{n},1)|^2 \leq C\frac{|r^{n}|^2}{Q(\vp^{n-1})^2} |(F'(\vp^{n-1}),1)|^2 + C \| g^{n-1} \|^2.
\]
If $|F'(s)| \sim |s|^p$ for some exponent $p \geq 1$, then applying a similar argument (omitting the lower order terms), we find that
\begin{align*}
\frac{|r^{n}|^2}{Q(\vp^{n-1})^2} |(F'(\vp^{n-1}),1)|^2 \leq C \frac{|r^{n}|^2}{Q(\vp^{n-1})^2} |(F(\vp^{n-1}),1)^{\frac{2p}{p+1}} \leq C|r^{n}|^2 Q(\vp^{n-1})^{2 \frac{p-1}{p+1}}.
\end{align*}
In order for this term to be controlled by the left-hand side of \eqref{stab:1}, we are forced to take $p = 1$, which limits $F$ to at most quadratic growth at infinity.

\begin{remark}
We mention that solutions to models with quadratic potentials (e.g. as limits of a SLI-based numerical scheme) may exhibit behavior different from those corresponding to models with super-quadratic potentials.  For instance, as noted in \cite[Remark 2.5]{Miran:longtime} for a Cahn--Hilliard tumor model (see \eqref{Tumor} below), establishing the long-time behavior of weak solutions seem to demand $F$ to have a polynomial growth at infinity faster than cubic.
\end{remark}

\section{Convergence of the time discrete solutions}\label{sec:conv}
In this section we show a weak solution to \eqref{CH} can be obtained in the limit $\tau \to 0$.  As the spatial discretization can be done in a standard way, we omit the details and refer to recent works \cite{GaTr,Metzger} for the convergence analysis of fully discrete solutions. For fixed $\tau$, we introduce the affine linear and piecewise constant extensions of time discrete functions $a^n$, $n =1, \dots, N_\tau$:
\begin{equation}\label{extension:def}
\begin{aligned}
a^\tau(t) := \frac{t-t^{n-1}}{\tau} a^{n} + \frac{t^{n} - t}{\tau} a^{n-1} & \text{ for } t \in [t^{n-1}, t^{n}], \\
a^{\tau,+}(t) := a^{n}, \quad a^{\tau,-}(t) := a^{n-1} & \text{ for } t \in \tc{(}t^{n-1}, t^{n}],
\end{aligned}
\end{equation}
for $a \in \{\vp, \mu, r, q\}$, \tc{where we set $\vp^\tau(0) = \vp_0$ and $q^\tau(0) = Q(\vp_0)$}.  Upon noting that 
\begin{equation}\label{inter:id}
\begin{aligned}
a^{\tau} (t)- a^{\tau,-}(t) = \frac{t-t^{n-1}}{\tau}(a^{n} - a^{n-1}) & \text{ for } t \in \tc{(}t^{n-1}, t^{n}], \\
a^{\tau,+} (t)- a^{\tau}(t) = \frac{t^{n}-t}{\tau}(a^{n} - a^{n-1}) & \text{ for } t \in \tc{(}t^{n-1}, t^{n}],
\end{aligned}
\end{equation}
we deduce from Theorem~\ref{thm:stab} the following uniform estimates (where the notation $a^{\tau,(\pm)}$ is a shorthand for $\{a^\tau, a^{\tau,+}, a^{\tau,-}\}$, while $a^{\tau, \pm}$ is a shorthand for $\{a^{\tau,+},a^{\tau,-}\}$):
\begin{equation}\label{unif:est}
\begin{aligned}
& \|\vp^{\tau, (\pm)} \|_{L^\infty(0,T;H^1)}^2 + \| \vp^{\tau, (\pm)} \|_{L^4(0,T;H^2)}^4+ \|q^{\tau,(\pm)} \|_{L^\infty(0,T)}^2 + \| r^{\tau, +} \|_{L^\infty(0,T)}^2 \\
& \qquad + \frac{1}{\tau} \| r^{\tau,+} - q^{\tau,-} \|_{L^2(0,T)}^2  + \|\mu^{\tau,+}\|_{L^2(0,T;H^1)}^2 + \|\pd_t \vp^{\tau} \|_{L^2(0,T;(H^1)^*)}^2 \\
& \qquad + \frac{1}{\tau} \| q^{\tau} - q^{\tau,\pm} \|_{L^2(0,T)}^2   + \frac{1}{\tau} \| \vp^{\tau} - \vp^{\tau,\pm} \|_{L^2(0,T;H^1)}^2  \leq C,
\end{aligned}
\end{equation}
with the last two terms coming from taking the $L^2$-norm of \eqref{inter:id} for $a \in \{\vp, q\}$.  Furthermore, from \eqref{timediff}, we infer that
\begin{align}
\tc{\tau^2 \| \pd_t \vp^\tau \|_{L^2(0,T;L^2)}^2 = \| \vp^{\tau,+} -\vp^{\tau,-} \|_{L^2(0,T;L^2)}^2} & \tc{\leq C \tau^{\frac{3}{2}},} \label{vp:pm:tau} \\
\| \vp^{\tau, (\pm)}(\cdot + h\tau) - \vp^{\tau, (\pm)} \|_{L^2(0,T-h\tau;L^2)}^2 & \leq Ch \tau \quad \text{ for any } h \in \{1, \dots, N_\tau\}.  \label{Simon}
\end{align}
Then, for arbitrary $\phi \in L^2(0,T;H^1(\Omega))$ testing \eqref{dis:1} and \eqref{dis:2} with $\phi$, summing from $n = 1$ to $N_\tau $, and using the above definitions, we arrive at
\begin{subequations}\label{disc:weak}
\begin{alignat}{2}
0 & = \int_0^T (\pd_t \vp^\tau - f^{\tau}, \phi) + (m(\vp^{\tau,-}) \nabla \mu^{\tau,+}, \nabla \phi) \dt , \label{dis:weak:1} \\
0 & = \int_0^T \eps (\nabla \vp^{\tau,+}, \nabla \phi) + \frac{r^{\tau,+}}{\eps Q(\vp^{\tau,-})} (F'(\vp^{\tau,-}), \phi) - (\mu^{\tau, +} + g^{\tau}, \phi) \dt. \label{dis:weak:2}
\end{alignat}
\end{subequations}
The main \tc{difference} compared to the convergence analysis of the standard SAV scheme \cite{SAVconv} is that the left-hand side of \eqref{dis:3} does not translate well into a time derivative of $r^\tau$ or $q^\tau$. \tc{However, as we note below, for strict RSAV schemes (i.e., $\zeta_n^\tau <1$ in \eqref{dis:4}) it suffices to show that the limits of $r^{\tau,+}$ and $Q(\varphi^{\tau,-})$ coincide, and equation \eqref{dis:3} does not play a role.}

\tc{
In the following we explore the convergence analysis for two choices of $\zeta_n^\tau$. As the choices are uniform in $n$ we use the notation $\zeta^\tau$. The first considers $\zeta^\tau = \zeta$ with a fixed $\zeta \in [0,1]$, covering both the standard SAV scheme $q^n = r^n$ with $\zeta = 1$, and the idealized update approach $q^n = Q(\vp^n)$ with $\zeta = 0$.  Note that aside from $\zeta = 1$, it is unknown whether for other values $\zeta \in [0,1)$ the criterion \eqref{dis:5} is fulfilled for all $n = 1, \dots, N_\tau$. This motivates the second choice where we consider $\zeta^\tau = 1 - e_\tau \tau$, with a non-negative sequence $\{e_\tau\}_{\tau > 0}$. This results in a method that asymptotically close to the standard SAV scheme with the relaxation effect almost completely negated for small time steps. Recalling the discussion before Remark~\ref{rmk:ideal}, we can pick $e_\tau = \max_n (1-w_n^\tau)$, so that $\zeta^\tau = 1 - e_\tau \tau \in \V_n^\tau$ for all $n = 1, \dots, N_\tau$.}


\begin{theorem}[Convergence]\label{thm:conv}
\tc{For any $\tau \in (0,\tau_*)$,} let $\{(\vp^k, \mu^k, r^k, q^k)\}_{k=1}^{N_\tau}$ denote the time discrete solutions to the relaxed SAV scheme \eqref{CH:dis}-\eqref{dis:4}, \tc{where we assume that either a fixed $\zeta \in [0,1]$ belongs to $\V_n^\tau$ for all $n = 1, \dots, N_\tau$, or there exists a non-negative convergent sequence $\{e_\tau\}_{\tau > 0}$ with $\lim_{\tau \to 0} e_\tau = e_* \in [0,\infty)$ such that $1-e_\tau \tau \in \V_n^\tau$ for all $n = 1,\dots, N_\tau$}. Then, there exists a non-relabelled subsequence $\tau \to 0$ and limit functions
\begin{align*}
\vp & \in L^\infty(0,T;H^1(\Omega)) \cap H^1(0,T;(H^1(\Omega))^*) \cap L^4(0,T;H^2(\Omega)), \\
 \mu & \in L^2(0,T;H^1(\Omega)), \quad q \in L^\infty(0,T),
\end{align*}
such that the interpolation functions $\{\vp^{\tau, (\pm)}, \mu^{\tau,+}, r^{\tau,+}, q^{\tau, (\pm)}\}$ satisfy the following compactness assertions for any $s < \infty$ in $d = 2$ and $s < 6$ in $d = 3$:
\begin{align*}
\vp^{\tau, (\pm)} \to \vp & \text{ weakly* in } L^\infty(0,T;H^1(\Omega)) \cap L^4(0,T;H^2(\Omega)), \\
\pd_t \vp^\tau \to \pd_t \vp & \text{ weakly in } L^2(0,T;(H^1(\Omega))^*), \\
\vp^{\tau,(\pm)} \to \vp & \text{ strongly in } L^2(0,T;L^s(\Omega)) \text{ and a.e.~in } \Omega \times (0,T),\\
\mu^{\tau,+} \to \mu & \text{ weakly in } L^2(0,T;H^1(\Omega)), \\
q^{\tau, (\pm)} \text{ and } r^{\tau, +} \to q & \text{ weakly* in } L^\infty(0,T), \\
\tc{Q(\vp^{\tau,(\pm)}) \to Q(\vp)} & \tc{\text{ strongly in } L^2(0,T)}.
\end{align*}
The limit pair $(\vp, \mu)$ is a weak solution to \eqref{CH} in the sense
\begin{subequations}\label{CH:weak}
\begin{alignat}{2}
\inn{\pd_t \vp}{\phi} + (m(\vp) \nabla \mu, \nabla \phi) - (f, \phi) & = 0, \label{weak:1}\\
\eps(\nabla \vp, \nabla \phi) + \eps^{-1} (F'(\vp), \phi) - (\mu + g, \phi) & = 0, \label{weak:2}
\end{alignat}
\end{subequations}
for all $\phi \in H^1(\Omega)$ and for a.e.~$t \in (0,T)$, while $q(t) = Q(\vp(t))$ holds for a.e.~$t \in (0,T)$.
\end{theorem}

\begin{proof}
The weak/weak* convergences of $\vp^{\tau,(\pm)}$, $\pd_t \vp^\tau$, $\mu^{\tau,+}$, $q^{\tau,(\pm)}$ and $r^{\tau,+}$ follow from standard compactness results in Bochner spaces. Then, from \eqref{unif:est} we infer
\begin{align}\label{str:conv:tau}
\| r^{\tau,+} - q^{\tau,-} \|_{L^2(0,T)}^2 + \| q^\tau - q^{\tau, \pm} \|_{L^2(0,T)}^2 + \| \vp^\tau - \vp^{\tau, \pm} \|_{L^2(0,T;H^1)}^2 \leq C \tau,
\end{align}
which yields the uniqueness of weak limits for $\{\vp^{\tau,(\pm)}\}$ and for $\{q^{\tau, (\pm)}, r^{\tau,+} \}$. The strong convergence \tc{of $\vp^{\tau,(\pm)}$} in $L^2(0,T;L^s(\Omega))$ comes from the application of \cite[\S 8, Thm.~5]{Simon} with the uniform boundedness of $\vp^{\tau,(\pm)}$ in $L^\infty(0,T;H^1(\Omega))$, the compact embedding $H^1(\Omega) \subset \subset L^s(\Omega)$ and the uniform time translation estimate \eqref{Simon}. \tc{
Let $F^{(0)}(s) = F(s)$ and $F^{(1)}(s) = F'(s)$. Then, a short calculation with the estimate
\[
|F^{(i)}(s_1) - F^{(i)}(s_2)| \leq C \big ( 1 + |s_1|^{3-i} + |s_2|^{3-i} \big )|s_1 - s_2| \text{ for } i = 0,1,
\]
together with the boundedness of $\vp^{\tau,(\pm)}, \vp$ in $L^4(0,T;H^2(\Omega)) \cap L^\infty(0,T;H^1(\Omega))$  yields $F'(\vp^{\tau,(\pm)}) \in L^\infty(0,T;L^2(\Omega))$ and
\begin{equation}\label{F:diff:conv}
\begin{alignedat}{2}
\| F'(\vp^{\tau,(\pm)}) - F'(\vp) \|_{L^{\frac{6}{5}}} & \leq C \big ( 1 + \| \vp^{\tau,(\pm)} \|_{L^6}^2 + \| \vp \|_{L^6}^2 \big ) \| \vp^{\tau,(\pm)} - \vp \| \leq C\| \vp^{\tau,(\pm)} - \vp \|, \\
 \| F(\vp^{\tau,(\pm)}) - F(\vp)\|_{L^1} & \leq C \big ( 1 + \| \vp^{\tau,(\pm)} \|_{H^1}^3 + \| \vp \|_{H^1}^3 \big ) \| \vp^{\tau,(\pm)} - \vp \| \leq C\| \vp^{\tau,(\pm)} - \vp \|.
\end{alignedat}
\end{equation}
Together with the following estimate
\[
|Q(\vp^{\tau,(\pm)}) - Q(\vp)| \leq \frac{1}{\eps} \frac{\| F(\vp^{\tau,(\pm)}) - F(\vp) \|_{L^1}}{ |Q(\vp^{\tau,(\pm)}) + Q(\vp)|} \leq C \| \vp^{\tau,(\pm)} - \vp \|,
\]
we deduce the strong convergence $Q(\vp^{\tau,(\pm)}) \to Q(\vp)$ in $L^2(0,T)$.
}

Then, using the boundeness of $m$ and the a.e.~convergence of $\vp^{\tau,-}$ we can infer that $m(\vp^{\tau,-}) \nabla \zeta \to m(\vp) \nabla \zeta$ strongly in $L^2(0,T;L^2(\Omega))$.  Thus, passing to the limit $\tau \to 0$ in \eqref{dis:weak:1} leads to \eqref{weak:1}, and by the continuous embedding $L^2(0,T;H^1(\Omega)) \cap H^1(0,T;(H^1(\Omega))^*) \subset C^0([0,T];L^2(\Omega))$, we have the attainment of the initial condition $\vp(0) = \vp_0$. In the case where $f = \hat{f}(\vp)$ is a bounded continuous function of $\vp$, we replace $f^{\tau}$ in \eqref{dis:weak:1} with $f(\vp^{\tau,-})$. Then, invoking the a.e.~convergence of $\vp^{\tau,-}$ in $\Omega \times (0,T)$ and the dominated convergence theorem yield that $\int_0^T (f(\vp^{\tau,-}), \phi) \dt \to \int_0^T (f(\vp), \phi) \dt$.

Passing to the limit $\tau \to 0$ in the terms of \eqref{dis:weak:2} other than the potential term is straightforward, and so we omit the details.  
For arbitrary $\phi \in L^2(0,T;H^1(\Omega))$ we have $\frac{(F'(\vp), \phi)}{Q(\vp)} \in L^1(0,T)$ and thus \tc{by \eqref{F:diff:conv}}
\begin{equation}\label{r:F:Q}
\begin{alignedat}{2}
& \left | \int_0^T \frac{r^{\tau,+} (F'(\vp^{\tau,-}), \phi)}{\eps Q(\vp^{\tau,-})} \mp \frac{r^{\tau,+}  (F'(\vp^{\tau,-}), \phi)}{\eps Q(\vp)} \mp \frac{r^{\tau,+} (F'(\vp), \phi) }{\eps Q(\vp)}  - \frac{q(F'(\vp), \phi)}{\eps Q(\vp)} \tc{\dt} \right | \\
& \quad  \leq \int_0^T \frac{|(F'(\vp^{\tau,-}), \phi)|}{\eps Q(\vp^{\tau,-}) Q(\vp)} \frac{|r^{\tau,+}|}{Q(\vp^{\tau,-}) + Q(\vp)} \| F(\vp^{\tau,-}) - F(\vp) \|_{L^1} \dt \\
& \qquad + \int_0^T \frac{|r^{\tau,+}|}{\eps Q(\vp)} \| \phi \|_{L^6} \| F'(\vp^{\tau,-}) - F'(\vp) \|_{L^{\frac{6}{5}}} \dt + \left | \int_0^T (r^{\tau,+} - q) \frac{(F'(\vp), \phi)}{\eps Q(\vp)}   \dt\right | \\
& \quad \leq C \| \vp^{\tau,-} - \vp \|_{L^2(0,T;L^2)} \| \phi \|_{L^2(0,T;H^1)} + \Big | \int_0^T (r^{\tau,+} - q) \frac{(F'(\vp), \phi)}{\eps Q(\vp)}  \dt \Big | \to 0,
\end{alignedat}
\end{equation}
as $\tau \to 0$, where we have applied the lower bound on $Q$. This implies that 
\begin{align}\label{weak:lim:F}
\lim_{\tau \to 0} \int_0^T \frac{r^{\tau,+}}{\eps Q(\vp^{\tau,-})} (F'(\vp^{\tau,-}), \phi) \dt = \int_0^T \frac{q}{\eps Q(\vp)} (F'(\vp), \phi) \dt.
\end{align}
It now remains to show $q = Q(\vp)$, so that we recover \eqref{weak:2} completely. \tc{We first consider the case $\zeta^\tau = \zeta \in [0,1)$, i.e., for strict RSAV schemes. With the strong convergence $Q(\vp^{\tau,+}) \to Q(\vp)$ in $L^2(0,T)$ and the weak$*$ convergence $r^{\tau,+} \to q$ in $L^\infty(0,T)$, we see that 
\[
q^{\tau,+} = \zeta r^{\tau,+} + (1-\zeta) Q(\vp^{\tau,+}) \to \zeta q + (1-\zeta) Q(\vp) \text{ weakly in } L^2(0,T).
\]
On the other hand, the weak$*$ convergence $q^{\tau,+} \to q$ in $L^\infty(0,T)$ yields the identity 
\[
q = \zeta q + (1-\zeta)Q(\vp),
\]
which implies $q = Q(\vp)$ as $1-\zeta > 0$. Then, we recover \eqref{weak:2} and the limit pair $(\vp,\mu)$ constitutes a weak solution to \eqref{CH}.
}

\tc{
We treat the remaining cases $\zeta^\tau = 1$ and $\zeta^\tau = 1 - e_\tau \tau$ together, where the former can be obtained from the latter by setting $e_\tau = e_* = 0$}. Adding and subtracting $q^n$ to \eqref{dis:3}, dividing by $\tau$ and testing with an arbitrary test function $\kappa \in H^1(0,T)$ such that $\kappa(T) = 0$. This yields after an integration by parts:
\begin{align}\label{zeta:1:lim}
\notag & \int_0^T (Q(\vp_0)-q^{\tau}) \kappa' \dt - \int_0^T \kappa \frac{ (F'(\vp^{\tau,-}), \pd_t \vp^{\tau})}{2 \eps Q(\vp^{\tau,-})}  \dt \\
& \quad = \int_0^T \kappa \frac{q^{\tau,+} - r^{\tau,+}}{\tau} \dt \tc{ \, = e_\tau \int_0^T \kappa  (Q(\vp^{\tau,+}) - r^{\tau,+}) \dt},
\end{align}
\tc{where we used that $q^\tau(0) = Q(\vp_0)$. Arguing as in \cite[Proof of Thm.~4.3]{Metzger}, we obtain 
\begin{equation}\label{F':Q:kappa}
\begin{alignedat}{2}
\int_0^T \kappa \frac{ (F'(\vp^{\tau,-}), \pd_t \vp^{\tau})}{2 \eps Q(\vp^{\tau,-})}   \dt & =  \int_0^T \frac{ \kappa (F'(\vp^{\tau,-}) ,\pd_t \vp^\tau )}{2\eps Q(\vp^{\tau,-}) Q(\vp^\tau)} \Big ( \frac{(F(\vp^{\tau}) - F(\vp^{\tau,-}), 1)}{Q(\vp^{\tau,-})+Q(\vp^\tau)} \Big ) \dt \\
&\quad + \int_0^T \kappa \frac{(F'(\vp^{\tau,-}) - F'(\vp^\tau), \pd_t \vp^\tau)}{2\eps Q(\vp^{\tau})} \dt \\
& \quad + \int_0^T \kappa \frac{d}{dt} Q(\vp^\tau) \dt.
\end{alignedat}
\end{equation}
Using the strong convergence of $Q(\vp^{\tau})$ in $L^2(0,T)$, the third term on the right-hand side of \eqref{F':Q:kappa} can be treated as follows after an integration by parts:
\[
\int_0^T \kappa \frac{d}{dt} Q(\vp^\tau) \dt = - \int_0^T \kappa' Q(\vp^\tau) \dt - \kappa(0) Q(\vp^\tau(0)) \to - \int_0^T \kappa' Q(\vp) \dt - \kappa(0) Q(\vp_0).
\]
With the help of \eqref{vp:pm:tau} and \eqref{str:conv:tau}, the second term on the right-hand side of \eqref{F':Q:kappa} can be estimated as
\begin{align*}
& \left | \int_0^T \kappa \frac{(F'(\vp^{\tau,-}) - F'(\vp^\tau), \pd_t \vp^\tau)}{2\eps Q(\vp^{\tau})} \dt \right | \leq C\|\kappa \|_{L^\infty(0,T)} \int_0^T \| F'(\vp^{\tau,-}) - F'(\vp^{\tau}) \| \| \pd_t \vp^\tau \| \dt \\
&  \quad \leq C \int_0^T \Big (1 + \| \vp^{\tau,-} \|_{L^6}^2 + \| \vp^\tau \|_{L^6}^2 \Big ) \| \vp^{\tau,-} - \vp^\tau \|_{L^6} \| \pd_t \vp^\tau \| \dt \\
& \quad \leq C \| \vp^{\tau,-} - \vp^{\tau} \|_{L^2(0,T;H^1)} \| \pd_t \vp^\tau \|_{L^2(0,T;L^2)} \leq C \tau^{\frac{1}{4}} \to 0 \quad \text{ as } \tau \to 0.
\end{align*}
Similarly, the first term on the right-hand side of \eqref{F':Q:kappa} can be estimated as
\begin{align*}
& \left |  \int_0^T \frac{ \kappa (F'(\vp^{\tau,-}) ,\pd_t \vp^\tau )}{2\eps Q(\vp^{\tau,-}) Q(\vp^\tau)} \Big ( \frac{(F(\vp^{\tau}) - F(\vp^{\tau,-}), 1)}{Q(\vp^{\tau,-})+Q(\vp^\tau)} \Big ) \dt  \right | \\
& \quad \leq C \| \kappa \|_{L^\infty(0,T)} \| F'(\vp^{\tau,-}) \|_{L^\infty(0,T;L^2)} \| \pd_t \vp^\tau \|_{L^2(0,T;L^2)} \| F(\vp^{\tau,-}) - F(\vp^\tau) \|_{L^2(0,T;L^1)}  \\
& \quad \leq C \| \pd_t \vp^\tau \|_{L^2(0,T;L^2)} \| \vp^{\tau,-} - \vp^\tau \|_{L^2(0,T;L^2)} \leq C \tau^{\frac{1}{4}} \to 0 \quad \text{ as } \tau \to 0.
\end{align*}
Then, passing to the limit in \eqref{zeta:1:lim} yields
\[
\int_0^T( \kappa' - e_* \kappa) (Q(\vp) - q) \dt =  0
\]
holding for arbitrary $\kappa \in H^1(0,T)$ with $\kappa(T) = 0$.  Choosing $\kappa$ as the solution to the ordinary differential equation $\kappa' - e_* \kappa = Q(\vp) - q$ with terminal condition $\kappa(T) = 0$ yields that 
\[
\| Q(\vp) - q \|_{L^2(0,T)} = 0,
\]
which provides the identification $q = Q(\vp)$. Hence, in the case $\zeta^\tau = 1 - e_\tau \tau$ we also recover \eqref{weak:2} with the limit pair $(\vp,\mu)$ constituting a weak solution to \eqref{CH}.  This completes the proof.
}
\end{proof}

\section{Numerical discussions}\label{sec:num}
\subsection{Fully discrete finite element approximation}\label{sec:FE}
For a bounded, convex domain $\Omega \subset \R^d$, $d \in \{1,2,3\}$, let $\{\Th\}_{h > 0}$ denote a regular family of conformal quasiuniform triangulations that partition $\Omega$ into disjoint open simplices $K$ such that $\max_{K \in \Th} \text{diam}(K) \leq h$. Let $\Sh$ denote the finite element space of continuous and piecewise linear functions:
\[
\Sh :=  \big \{ q_h \in C^0(\overline{\Omega}) \, : \, q_h \vert_K \in \mathcal{P}_1 \; \forall K \in \Th \big \} \subset H^1(\Omega), \quad \tc{U_h := \text{dim}(\Sh)},
\]
with the set of basis functions $\{ \chi_{h,k} \}_{k = 1}^{\tc{U_h}}$ that forms a dual basis to the set of vertices $\{\bm{x}_k\}_{k=1}^{\tc{U_h}}$ of $\Th$. The nodal interpolation operator $\Ih : C^0(\overline{\Omega}) \to \Sh$ is defined by $(\Ih \eta)(\bm{x}_k) = \eta(\bm{x}_k)$ for all $k = 1, \dots, \tc{U_h}$, and we introduce the semi-inner product on $C^0(\overline{\Omega})$:
\[
(\eta_1, \eta_2)^h := \int_\Omega \Ih(\eta_1 \eta_2) \dx = \sum_{k=1}^{\tc{U_h}} \eta_1(\bm{x}_k) \eta_2(\bm{x}_k) \int_\Omega \chi_{h,k}(\bm{x}) \dx.
\]
If $f = f(t,x)$ and $g = g(t,x)$ are given functions, they are approximated with $f^k_h(\cdot) = \Ih^c(f^k(\cdot))$ and $g^k_h(\cdot) = \Ih^c(g^k(\cdot))$ for $k = 1, \dots, N_\tau$, via the Cl\'ement operator $\Ih^c : L^2(\Omega) \to \Sh$ \cite{Clement}. If they are functions of $\vp$, we take $f_h^{k} = f(\vp_h^{k})$ and $g_h^{k} = g(\vp_h^{k})$. 

For more regular initial data $\vp_0 \in H^2(\Omega)$ we consider the discrete initial data $\vp^0_h := \Ih(\vp_0) \in \Sh$, otherwise we take $\vp^0_h = \Ih^c(\vp_0) \in \Sh$, with $q^0_h := Q_h(\vp^0_h) = (( \eps^{-1} F(\vp^0_h), 1)^h + C_0)^{1/2}$.  The fully discrete finite element scheme for \eqref{CH:alt} reads as follows: Given data $\vp^0_h  \in \Sh$, $q^0_h \in \R$, and $\{f^{k}_h, g^k_h\}_{k=1}^{N_\tau}$, for $n = 1, \dots, N_\tau$, find discrete solutions $\vp^n_h, \mu^n_h \in \Sh$ and $r^n_h \in \R$ that satisfy for all $j = 1, \dots, \tc{U_h}$:
\begin{subequations}\label{full:diss}
\begin{alignat}{2}
0& = (\vp^{n}_h - \vp^{n-1}_h, \chi_{h,j})^h + \tau(\Ih(m(\vp^{n-1}_h)) \nabla \mu^{n}_h , \nabla \chi_{h,j} ) - \tau (f^{n-1}_h, \chi_{h,j})^h , \label{full:diss:1} \\
0& =  \eps (\nabla \vp^{n}_h, \nabla \chi_{h,j}) - (\mu^n + g^{n-1}_h, \chi_{h,j})^h + \frac{r^n_h}{\eps Q_h(\vp^{n-1}_h)} (F'(\vp^{n-1}_h), \chi_{h,j})^h,\label{full:diss:2} \\
0& =  r^n_h - q^{n-1}_h - \frac{1}{2 \eps Q_h(\vp^{n-1}_h)} (F'(\vp^{n-1}_h), \vp^{n}_h - \vp^{n-1}_h)^h, \label{full:diss:3}
\end{alignat}
\end{subequations}
and update $q^n_h  = \zeta_n^\tau r_h^n + (1-\zeta_n^\tau) Q_h(\vp^{\tc{n}}_h)$ with a constant $\zeta_n^\tau \in [0,1]$ such that the analogue of \eqref{dis:5} is fulfilled. This fully discrete scheme is linear with respect to the unknowns $(\vp^{n}_h, \mu^{n}_h, r^{n}_h, q^n_h) \in \Sh \times \Sh \times \R \times \R$. Let us define the following matrices
\[
\bM_{ij} = (\chi_{h,i}, \chi_{h,j})^h, \quad \bS_{ij} = (\nabla \chi_{h,i}, \nabla \chi_{h,j}), \quad (\bS_h^{n-1})_{ij} = (\Ih(m(\vp_h^{n-1})) \nabla \chi_{h,i}, \nabla \chi_{h,j}),
\]
where we note that $\bM$ is diagonal \tc{with $\bM_{ii} = (\chi_{h,i}, 1)$}; vectors $\bm{\vp}^{n}_h  := \bM^{-1} [(\vp^n_h, \chi_{h,k})^h ]_{k=1}^{\tc{U_h}} = \tc{(\vp_h^n(\bm{x}_1), \vp_h^n(\bm{x}_2), \dots, \vp_h^n(\bm{x}_{U_h}))^{\top}}$, likewise for $\bm{\mu}^{n}_h$, $\bm{\vp}^{n-1}_h$; vectors $\bm{f}^{n-1}_h := \bM^{-1} [(f^{n-1}_h, \chi_{h,k})^h]_{k=1}^{\tc{U_h}} = \tc{(f_h^{n-1}(\bm{x}_1), f_h^{n-1}(\bm{x}_2), \dots, f_h^{n-1}(\bm{x}_{U_h}))^{\top}}$, likewise for $\bm{g}^{n-1}_h$; and vector
\[
\bm{b}^{n-1} := \frac{F'( \bm{\vp}^{n-1}_h)}{\eps Q_h(\vp^{n-1}_h)}, \text{ where } Q_h(\vp^{n-1}_h) = \Big ( \frac{1}{\eps} (F(\bm{\vp}^{n-1}_h) , 1)^h + C_0 \Big )^{1/2},
\]
with the nonlinearities applied component-wise. Then, \eqref{full:diss} can be expressed as 
\begin{align*}
\bm{0} & = \bM (\bm{\vp}^n_h - \bm{\vp}^{n-1}_h  - \tau \bm{f}^{n-1}_h) + \tau \bS^{n-1}_h (\eps \bM^{-1} \bS \bm{\vp}^n_h + r^n \bm{b}^{n-1} - \bm{g}^{n-1}_h), \\
0 & = r^n_h - q^{n-1}_h - \frac{1}{2} \bM \bm{b}^{n-1} \cdot (\bm{\vp}^{n}_h - \bm{\vp}^{n-1}_h).
\end{align*}
Defining the invertible matrix $\mathbb{B}_{n-1} :=  \mathbb{I} + \eps \tau \bM^{-1} \bS^{n-1}_h \bM^{-1} \bS$ with identity matrix $\mathbb{I}$, substituting the second equation into the first equation, applying $\mathbb{B}_{n-1}^{-1}$ and taking the product with $\bM\bm{b}^{n-1}$ would first yield an expression for $\tc{d^{n-1} = \,} \bM \bm{b}^{n-1} \cdot \bm{\vp}^n_h$ as
\begin{align*}
d^{n-1} & := \bM \bm{b}^{n-1} \cdot \bm{\vp}^n_h = \frac{\bM \bm{b}^{n-1} \cdot \mathbb{B}_{n-1}^{-1} \bm{c}^{n-1}}{1 + \frac{\tau}{2} \bM \bm{b}^{n-1} \cdot (\mathbb{B}_{n-1}^{-1} \bM^{-1} \bS^{n-1}_h \bm{b}^{n-1})}, \\
\bm{c}^{n-1} & := \bm{\vp}^{n-1}_h + \tau \bm{f}^{n-1}_h + \tau \bM^{-1} \bS^{n-1}_h (\bm{g}^{n-1}_h + [\tfrac{1}{2} \bM \bm{b}^{n-1} \cdot \bm{\vp}^{n-1}_h - q^{n-1}_h] \bm{b}^{n-1} ).
\end{align*}
Then, $\bm{\vp}^n_h$, $r^n_h$, $\bm{\mu}^n_h$ and $q^n_h$ can be computed via
\begin{equation*}
\begin{alignedat}{3}
\bm{\vp}^{n}_h & = \mathbb{B}_{n-1}^{-1}\bm{c}^{n-1} - \frac{\tau}{2} d^{n-1} \mathbb{B}_{n-1}^{-1} \bM^{-1} \bS^{n-1}_h \bm{b}^{n-1}, \quad && r^n_h = q^{n-1}_h + \frac{1}{2} \bM\bm{b}^{n-1} \cdot (\bm{\vp}^{n}_h - \bm{\vp}^{n-1}_h), \\
\bm{\mu}^n_h & = \eps \bM^{-1} \bS \bm{\vp}^n_h - \bm{g}^{n-1}_h + r^n_h \bm{b}^{n-1}, \quad && q^n_h = \zeta_n^\tau r^n_h + (1-\zeta_n^\tau) Q_h(\vp^n_h).
\end{alignedat}
\end{equation*}
\tc{As noted in \cite{SAV1}, the main computational cost in each step amounts to solving two linear systems involving $\mathbb{B}_{n-1}$. Once $\mathbb{B}_{n-1}^{-1} \mathbb{M}^{-1} \mathbb{S}_{h}^{n-1} \bm{b}^{n-1}$ and $\mathbb{B}_{n-1}^{-1} \bm{c}^{n-1}$ are computed, the constant $d^{n-1}$ and update $\bm{\vp}_h^{n}$ do not involve any matrix inverse operations.}

\subsection{Convergence test}
In this section we report on numerical \tc{tests} for the scheme \eqref{CH:dis} on an interval $\Omega = [0,1]$ discretized with equidistant nodal points and spatial step size $h = 0.01$. \tc{Let us first explore the optimal values of $\zeta_n^\tau$ as defined in \eqref{RSAV:opt} for various values of $\eta$ and $M$.  We choose} the source function $f$ so that the analytical solution to \eqref{CH} with $m(\vp) = 1$, $g = 0$, $F'(\vp) = \vp^3 - \vp$ and $\eps = 1$ is
\[
\vp^*(x,t) = \cos(\pi x) (1+t).
\]
We fix $C_0 = 1$, $\tau = 0.01$ and $N_\tau = 500$ (so that $T = N_\tau \tau = 5$). In our experiments we observe that the discrete energy $\frac{1}{2} \bS \bm{\vp}^n_h \cdot \bm{\vp}^n_h + Q(\bm{\vp}_h^n)^2 - C_0$ is increasing in time, \tc{see Figure~\ref{fig:num:GL}(a)}, and thus the equation does not exhibit a dissipative structure involving the Ginzburg--Landau functional.
\begin{figure}[htbp]
\centering
\subfloat[$\vp^*(x,t) = \cos(\pi x) (1+t)$]{
\includegraphics[width=5cm]{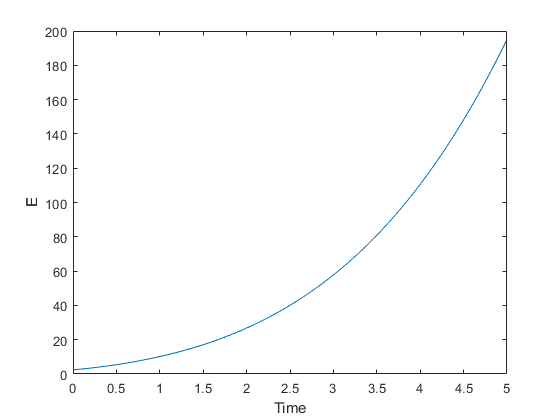}
}%
\subfloat[$\vp^*(x,t) =  \exp(\cos(\pi x)) \cos(t)$]{
\includegraphics[width=5cm]{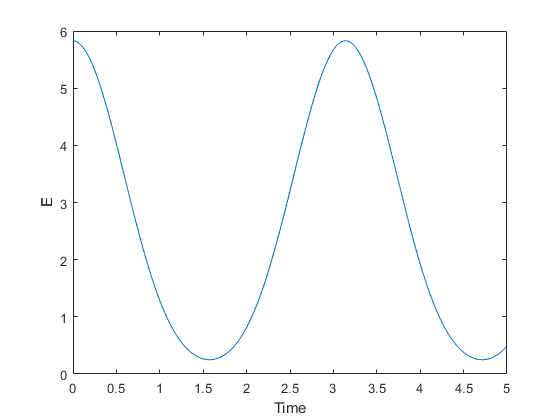}
}%
\subfloat[$\vp^*(x,t) = \exp(\cos(t))(\sin(\pi x))^2$]{
\includegraphics[width=5cm]{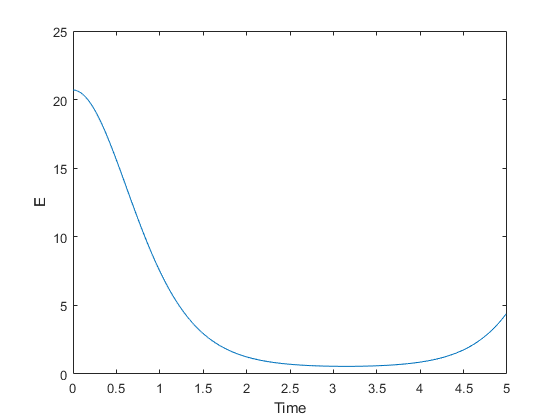}
}%
\caption{A plot of the discrete Ginzburg--Landau energy $E = \frac{1}{2} \bS \bm{\vp}^n_h \cdot \bm{\vp}^n_h + Q(\bm{\vp}_h^n)^2 - C_0$ for the numerical solutions approximating the three analytical test solutions.}
\label{fig:num:GL}
\end{figure}

For $M, \eta \sim O(10^{-3})$, the optimal $\zeta_n^\tau$ is $0$ for all $n = 1, \dots, N_\tau$. Moreover, we observed a switching behavior of the optimal $\zeta_n^\tau$ when either one of the parameter is zero, and the other parameter is small.  For example, with $\eta = 0$ the optimal $\zeta_n^\tau$ is $0$ for all $n = 1, \dots, N_\tau$ as long as $M \geq 0.05$. Decreasing $M$ yields a switching of the optimal value to $\zeta_n^\tau = 1$ for later iterations, and the smaller the value of $M$ the earlier the switching, see Figure~\ref{fig:etaM}(a) and (b). On the other hand, with $M = 0$ and $\eta = 10^{-5}$ the optimal $\zeta_n^\tau$ is zero except for the first iteration, similar to \cite[Fig.~2]{Zhang}, while with $\eta = 10^{-6}$ the optimal $\zeta_n^\tau$ hovers around 0.9 for $\sim 330$ iterations before switching to $0$, see Figure~\ref{fig:etaM}(c).

\begin{figure}[htbp]
\centering
\subfloat[$\eta = 0, M = 10^{-2}$]{
\includegraphics[width=5cm]{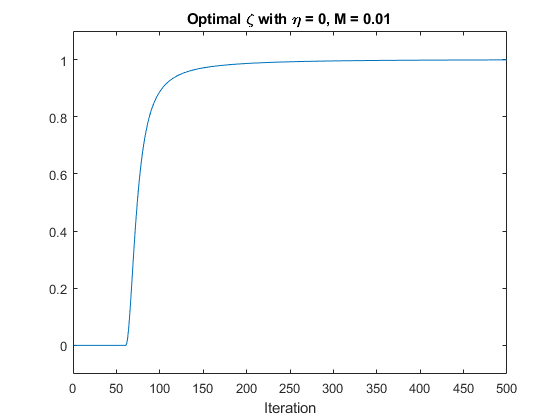}
}%
\subfloat[$\eta = 0, M = 10^{-3}$]{
\includegraphics[width=5cm]{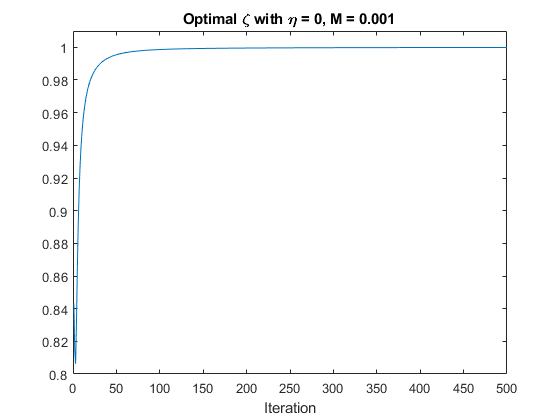}
}%
\subfloat[$\eta = 10^{-6}, M = 0$]{
\includegraphics[width=5cm]{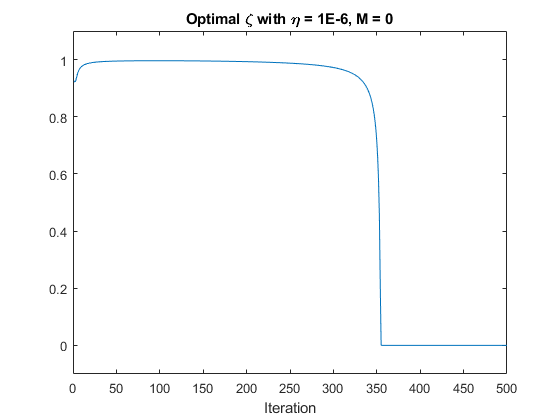}
}%
\caption{\tc{Switching behavior observed for the} optimal values of $\zeta_n^\tau$ computed via \eqref{RSAV:opt} for various values of $\eta$ and $M$.}
\label{fig:etaM}
\end{figure}

From our experiments it is interesting to see that the optimal $\zeta_n^\tau$ is 0 for moderate values of $M$ and $\eta$, meaning that it is possible to consider the \tc{idealized} update $q^n = Q(\vp^n)$ which preserves the consistency between the modified and original Ginzburg--Landau energy. It also suggests some form of numerical stability is available for the choice $q^n = Q(\vp^n)$, although this is not immediate from our \tc{stability} analysis. When $\eta$ and $M$ are sufficiently small or even zero, the tendency for the optimal $\zeta_n^\tau$ to be $1$ is supported by the fact that $R_n(1) \leq 0$ is satisfied even when $\eta = M = 0$.

Next we report on the $L^2(0,5;L^2(\Omega))$-error between the numerical solution and the analytical solution. The switching behavior for the optimal $\zeta_n^\tau$ for $\tau = 0.01$ displayed in Figure~\ref{fig:etaM} does not occur for smaller time steps with the same values of $\eta$ and $M$.  I.e., for $\tau = 0.0001$, $M= 0$ and $\eta = 10^{-6}$ the optimal $\zeta_n^\tau$ is $0$ for all $n = 1, \dots, N_\tau$. Hence, we compare the numerical errors in the \tc{$L^2(0,5;L^2(\Omega))$-norm with fixed $\zeta_n^\tau = \zeta \in \{0, 0.25, 0.5, 0.75, 1\}$ and a finer spatial discretization with step size $h = 0.001$. Since different time steps $\tau \in \{0.1, 0.05, 0.025, 0.0125, 0.00625, 0.003125, 0.0015625\}$ are used, with a terminal time $T = 5$ the number of iterations $N_\tau$ ranges from $100$ to $3200$. The comparison is displayed in Table~\ref{tbl:conv:1}, from which we observe the reduction of numerical error as the time step decreases for all considered settings. We note that all numerical values are one order of magnitude smaller to those reported in \cite[Fig.~4 (SAV/BDF1, R-SAV/BDF1)]{Zhang} for a standard Cahn--Hilliard equation discretized with a first order time discretization. While RSAV having smaller error values than SAV, unlike in \cite{Zhang}, the magnitude of the errors are almost indistinguishable among all cases, and the ratio of RSAV errors to SAV errors approaches 1 as the time step decreases. Thus, in this test case we do not see any clear advantage of RSAV over SAV in terms of numerical accuracy. Furthermore, the idealized choice $\zeta_n^\tau = 0$ may not always achieve the smallest error among RSAV schemes for certain time steps.}

We repeat our error comparison with two more examples, where the source functions are chosen specifically so that with $m(\vp) = 1$, $g = 0$, $F'(\vp) = \vp^3 - \vp$ and $\eps = 1$, the analytical solutions are
\[
\vp^*(x,t) = \exp(\cos(\pi x)) \cos(t) \quad \text{ and } \quad  \exp(\cos(t))(\sin(\pi x))^2,
\]
respectively. In Figure~\ref{fig:num:GL}(b) and (c), we plot the corresponding discrete Ginzburg--Landau energy for the above two analytical solutions with $h = \tau = 0.01$. Then, we set $h = 0.001$ in the error comparison between RSAV and SAV schemes for these two examples, which can be found in Tables~\ref{tbl:conv:2} and \ref{tbl:conv:3}, respectively. We again note that the modified Cahn--Hilliard equation does not exhibit a dissipative structure with the Ginzburg--Landau functional, RSAV have smaller (but similar in magnitude) errors than SAV, the ratio of RSAV errors to SAV errors approaches 1 as the time step decreases, and the idealized choice $\zeta_n^\tau = 0$ may not achieve the smallest errors among RSAV schemes for certain time steps.

Our findings indicate that in contrast to dissipative systems of the form \eqref{diss}, for Cahn--Hilliard equations with mass sources such as \eqref{CH}, RSAV schemes perform only marginally better than the standard SAV scheme in terms of numerical accuracy.

\begin{table} [!ht]
\begin{tabular}{|l||c|c|c|c|c|c|}
\hline
$(\tau, \zeta)$ & $1$ &  $0.75$ & $0.5$ & $0.25$ & $0$ \\
\hline
0.1 &  0.1235112595  & 0.1223488639 & 0.1221746557 & 0.1221087366 & {\bf 0.1220743681} \\
0.05 & 0.0657475964   & 0.0654207105 & 0.0653961915 &  0.0653875327 & {\bf 0.0653831138} \\
0.025  & 0.0370353831 & 0.0369482551 & 0.0369450147 &  0.0369439043  & {\bf 0.0369433434} \\
0.0125  & 0.0227271731 & 0.0227046023 & 0.0227041869 & 0.0227040466 &  {\bf 0.0227039746} \\
0.00625  & 0.0155923393  &  0.0155909899 & {\bf 0.0155845443} & 0.0156222078  & 0.0155860488  \\
0.003125 & 0.0120314862 & 0.0120300314 & 0.0120300093 & 0.0120300181& {\bf 0.0120300069} \\
0.0015625 & 0.0102552922 & 0.0102549247 & 0.0102549247 & {\bf 0.0102549230} & 0.0102549241 \\
\hline
\end{tabular}
\caption{Comparison of $L^2(0,5;L^2(\Omega))$-error for various choices of constant $\zeta$ with $\eta = 0.95$ and $M = 1$. The optimal $\zeta_n^\tau$ computed via \eqref{RSAV:opt} is $\zeta_n^\tau = 0$ in all scenarios. The smallest value in each row is highlighted in bold. The exact solution is $\vp^*(x,t) = \cos(\pi x)(1+t)$.}
\label{tbl:conv:1}
\end{table}

\begin{table} [!ht]
\begin{tabular}{|l||c|c|c|c|c|c|}
\hline
$(\tau, \zeta)$ & $1$ &  $0.75$ & $0.5$ & $0.25$ & $0$ \\
\hline
0.1 &  0.1580897113  & 0.1573562578 & 0.1573021781 & 0.1572764459 & {\bf 0.1572614308} \\
0.05 &  0.0785704830  & 0.0784190660 & 0.0784103517& 0.0784068449& {\bf 0.0784050015} \\
0.025  & 0.0391241631 & 0.0390891339 & 0.0390878945 & 0.0390874601& {\bf 0.0390872404} \\
0.0125  & 0.0195471355 & 0.0195389811 & {\bf 0.0195358830} & 0.0195387801 & 0.0195387550  \\
0.00625  & 0.0100802056  & 0.0100784782 & 0.0100784619 &0.0100784607 & {\bf 0.0100784602} \\
0.003125 & 0.0057771876 & {\bf 0.0057769012} & 0.0057769104& 0.0057769112 & 0.0057769075 \\
0.0015625 &  0.0042246587 & 0.0042246609 & {\bf 0.0042246027} & 0.0042246434& 0.0042246623 \\
\hline
\end{tabular}
\caption{Comparison of $L^2(0,5;L^2(\Omega))$-error for various choices of constant $\zeta$ with $\eta = 0.95$ and $M = 1$. The optimal $\zeta_n^\tau$ computed via \eqref{RSAV:opt} is $\zeta_n^\tau = 0$ in all scenarios. The smallest value in each row is highlighted in bold. The exact solution is $\vp^*(x,t) = \exp(\cos(\pi x)) \cos(t)$.}
\label{tbl:conv:2}
\end{table}

\begin{table} [!ht]
\begin{tabular}{|l||c|c|c|c|c|c|}
\hline
$(\tau, \zeta)$ & $1$ &  $0.75$ & $0.5$ & $0.25$ & $0$ \\
\hline
0.1 & 0.0743044854 & 0.0742724251 & 0.0742580052 & 0.0742503247  &  {\bf 0.0742457608} \\
0.05 & 0.0362455345 & 0.0362330117& 0.0362305278 &0.0362294871  & {\bf 0.0362289303}  \\
0.025  & 0.0175288174 & 0.0175263383  & 0.0175260131 & 0.0175258949 &  {\bf 0.0175258344} \\
0.0125  & 0.0085019366 & 0.0085014574  & 0.0085014301 & 0.0085014207 &  {\bf 0.0085014149} \\
0.00625  & 0.0046007411& 0.0046007312 & {\bf 0.0046007299} & 0.0046007307&  0.0046007395 \\
0.003125 & 0.0034729641 & 0.0034730059 & {\bf 0.0034729916} & 0.0034729941 & 0.0034729946  \\
0.0015625 & 0.0034123532 & 0.0034123725 & 0.0034123619 & {\bf 0.0034123592} & 0.0034123702  \\
\hline
\end{tabular}
\caption{Comparison of $L^2(0,5;L^2(\Omega))$-error for various choices of constant $\zeta$ with $\eta = 0.95$ and $M = 1$. The optimal $\zeta_n^\tau$ computed via \eqref{RSAV:opt} is $\zeta_n^\tau = 0$ in all scenarios. The smallest value in each row is highlighted in bold. The exact solution is $\vp^*(x,t) = \exp(\cos(t)) (\sin(\pi x))^2$.}
\label{tbl:conv:3}
\end{table}



\section{Applications}\label{sec:appl}
In this section \tc{we apply the SAV framework to various Cahn--Hilliard models that do not exhibit an obvious Lyapunov or dissipative structure due to the presence of source terms. In preliminary investigations we did not observe significant visual differences between the numerical solutions obtained from the standard SAV approach and the RSAV approach. Hence, to simplify the presentation,} we propose time discretizations based on the standard SAV approach and demonstrate their stability. \tc{The corresponding RSAV-based time discretizations and their stability can be obtained with a straightforward modification.} In turn, the convergence as $\tau \to 0$ to the corresponding weak solutions can be inferred similarly as in the proof of Theorem~\ref{thm:conv}. We then display some qualitative simulations on a square domain $\Omega = [0,1]^2$ discretized with a standard Friedrichs--Keller triangulation and spatial step size $h = 0.01$. Unless specified, in all simulations below we take the quartic potential $F(s) = \frac{1}{4}(s^2-1)^2$, and set constant $C_0 = 1$ in \eqref{scalarvar}.

\subsection{Microphase separation of diblock copolymer}
A phase field model proposed to describe the dynamics of microphase separation of diblock copolymers is the following Cahn--Hilliard--Oono equation \cite{Gior,Oono}
\begin{align}\label{CHO}
\pd_t \vp + \eta (\vp - c) = \Lap \mu, \quad \mu = - \eps \Lap \vp + \eps^{-1} F'(\vp),
\end{align}
furnished with homogeneous Neumann conditions. In the above $\eta >0 $ is a constant related to the chain length of the copolymer and $c \in \R$ is a prescribed relative mass average. If $c = \mean{\vp_0} = \frac{1}{|\Omega|} \int_\Omega \vp_0 \dx$ equals the mean value of the initial condition, then the system is also known as the Ohta--Kawasaki equation \cite{Ohta}, which is the conserved gradient flow of the Ohta--Kawaski functional
\[
E(\vp) = \int_\Omega \frac{\eps}{2} |\nabla \vp|^2 + \frac{1}{\eps} F(\vp) \dx + \int_{\Omega \times \Omega} \eta G(x-y) (\vp(x) - \mean{\vp})(\vp(y) - \mean{\vp}) \dx \dy,
\]
where $G$ is the Green function associated to the Neumann--Laplacian. A formal integration of the first equation in \eqref{CHO} yields a differential equation for the mean value $\mean{\vp}$:
\begin{align}\label{CHO:mass}
\frac{d \mean{\vp}}{dt} + \eta (\mean{\vp} - c) = 0 \quad \implies \quad \mean{\vp}(t) = c + e^{-\eta t}(\mean{\vp_0} - c) \; \forall t > 0,
\end{align}
where we note the conservation of mass $\mean{\vp}(t) = \mean{\vp_0}$ occurs only when $c = \mean{\vp_0}$.

A SAV-based time discretization of \eqref{CHO} reads as
\begin{subequations}\label{CHO:diss}
\begin{alignat}{2}
\vp^{n} - \vp^{n-1} & = \tau\, \Lap \mu^{n} + \tau \eta (c - \vp^{n-1}) \label{CHO:dis:1} \\
\mu^{n} & = - \eps \Lap \vp^{n} + \frac{q^{n}}{\eps Q(\vp^{n-1})} F'(\vp^{n-1}), \label{CHO:dis:2} \\
q^{n} - q^{n-1} & = \frac{1}{2 \eps Q(\vp^{n-1})} (F'(\vp^{n-1}), \vp^{n} - \vp^{n-1}), \label{CHO:dis:3}
\end{alignat}
\end{subequations}
furnished with homogeneous Neumann conditions. An interesting feature of this SAV discretization is that the source term need not be a bounded function, as we have the following discrete analogue of \eqref{CHO:mass}:
\[
\mean{\vp^n} = (1-\tau \eta)\mean{\vp^{n-1}} + \tau \eta c \quad \implies \quad \mean{\vp^n} = \mean{\vp_0}(1-\tau \eta)^n + c(1 - (1-\tau \eta)^n).
\]
For $\tau < \eta^{-1}$ we have the a priori estimate $|\mean{\vp^n}| \leq C$ for all $n \in \N$.  Then, testing \eqref{CHO:dis:1} with $\mu^n$ and $\vp^n$, \eqref{CHO:dis:2} with $\vp^n - \vp^{n-1}$, \eqref{CHO:dis:3} with $2q^n$, and upon summing we obtain for $G^k$ defined as in \eqref{def:G}:
\begin{align*}
& G^n - G^{n-1} + \frac{1}{2} \big ( \| \vp^n - \vp^{n-1} \|^2 + \eps \| \nabla(\vp^n - \vp^{n-1}) \|^2 \big ) + |q^n - q^{n-1}|^2 + \tau \| \nabla \mu^n \|^2 \\
& \quad =
\tau  \eta \Big ( (c - \mean{\vp^{n-1}}, \mu^n + \vp^n) + (\mean{\vp^{n-1}} - \vp^{n-1}, \mu^n+ \vp^n) \Big ) \\
& \quad \leq C \tau \big ( \| \mu^n\|_{L^1} + \| \vp^n \| \big ) + C \tau \| \nabla \vp^{n-1} \|( \| \nabla \mu^n \| + \| \vp^n \|),
\end{align*}
on account of the Poincar\'e inequality and
\[
(\mean{\vp^{n-1}} - \vp^{n-1}, \mu^n) = (\mean{\vp^{n-1}} - \vp^{n-1}, \mu^n - \mean{\mu^n}) \leq C \| \nabla \vp^{n-1} \| \| \nabla \mu^n \|.
\]
Upon using \eqref{mu:trick} and \eqref{mean:mu:est} with $m_0 = 1$, $g = 0$, $\zeta = 0$, $r^n = q^n$, we infer
\begin{equation}\label{stab:est:CHO}
\begin{aligned}
& G^n - G^{n-1} + \frac{\min(1,\eps)}{2}  \| \vp^n - \vp^{n-1} \|_{H^1}^2 +\frac{3}{4} |q^n - q^{n-1}|^2 + \frac{\tau}{2} \| \nabla \mu^n \|^2  \\
& \quad \leq C \tau (1 + G^{n-1} + G^n),
\end{aligned}
\end{equation}
which yields stability of the SAV scheme for $\tau$ sufficiently small. Notice that above calculation relies on the uniform control of the mean value of $\vp^n$, which is possible here due to the specific linear structure of the source term $f(\vp) = \eta(c - \vp)$.

We consider the scheme \eqref{CHO:diss} with parameter values $\eta = 0.001$ and $\eps = 0.01$, time step $\tau = 0.01$,  and $c = \mean{\vp^0}$ is the mean value of the initial condition $\vp^0$. Taking $\vp^0$ as the perturbation of $-0.5$ (Figure~\ref{fig:CHO1} top row) and of $-0.1$ (Figure~\ref{fig:CHO1} bottom row)  with uniformly distributed random numbers in $(0,0.2)$, we display the discrete solution $\vp^n$ at $n = 0$ (initial condition), $500$, $10000$ and $50000$ in Figure~\ref{fig:CHO1}. The behavior of the discrete solution for both settings is similar to the simulations in \cite[Sec.~3]{Kim}.

\subsection{Image segmentation}
In \cite{Yang}, the authors proposed an image segmentation model with the following Cahn--Hilliard equation
\begin{align}\label{Seg}
\pd_t \vp = \Lap \mu - \frac{\eta \Big (\lambda_1(I(x) - c_1)^2 - \lambda_2(I(x) - c_2)^2 \Big )}{\pi( \eta^2 + (\vp - 0.5)^2)}, \quad \mu = - \eps \Lap \vp + \eps^{-1} F'(\vp),
\end{align}
furnished with homogeneous Neumann conditions, $F(s) = 2s^2(s-1)^2$, positive constants $\lambda_1, \lambda_2$, grayscale image function $I$ rescaled to the range $[0,1]$, and $c_1$ and $c_2$ represent averaged pixel intensities inside and outside the segmented regions. The source term can be interpreted as a $L^2$-gradient of the energy functional
\[
M(u) = \lambda_1 \int_\Omega (I(x) - c_1)^2 H_{\eta}(u - \tfrac{1}{2}) \dx + \lambda_2 \int_\Omega (I(x) - c_2)^2 [1 - H_{\eta}(u - \tfrac{1}{2}) ]\dx
\]
where $H_\eta(z) = \frac{1}{2}(1 + \frac{2}{\pi} \arctan(z/\eta))$ is a $C^\infty$-regularization of the Heaviside function, and $\delta_\eta(z) = \eta/(\pi(\eta^2 + z^2)) = H_\eta'(z)$ approximates the Dirac measure. The average intensities $c_1$ and $c_2$ are defined as
\begin{align}\label{av:inten}
c_1(\vp) = \frac{(I, H_\eta(\vp - \frac{1}{2}))}{(1, H_\eta(\vp - \frac{1}{2}))}, \quad c_2 (\vp)=  \frac{(I, 1 - H_\eta(\vp - \frac{1}{2}))}{(1, 1 - H_\eta(\vp - \frac{1}{2}))}.
\end{align}

Due to the bounded source term, a SAV-based discretization of \eqref{Seg} reads as
\begin{equation}\label{dis:Seg}
\begin{aligned}
\vp^n - \vp^{n-1}  & = \tau \Lap \mu^n - \tau \frac{\eta \Big (\lambda_1(I(x) - c_1^{n-1})^2 - \lambda_2(I(x) - c_2^{n-1})^2 \Big )}{\pi( \eta^2 + (\vp ^{n-1}- 0.5)^2)}, \\
\mu^{n} & = - \eps \Lap \vp^{n} + \frac{q^{n}}{\eps Q(\vp^{n-1})} F'(\vp^{n-1}), \\
q^{n} - q^{n-1} & = \frac{1}{2 \eps Q(\vp^{n-1})} (F'(\vp^{n-1}), \vp^{n} - \vp^{n-1}),
\end{aligned}
\end{equation}
furnished with homogeneous Neumann conditions. Then, the average intensities are updated to $c_1^n := c_1(\vp^n)$ and $c_2^n := c_2(\vp^n)$ via \eqref{av:inten}. It is straightforward to infer the estimate \eqref{stab:est:CHO} for $G^k$, leading to stability of the SAV scheme for sufficiently small $\tau$. 

We apply \eqref{Seg} to segment images of a cow (Figure~\ref{fig:seg} top row) and of a blood vessel (Figure~\ref{fig:seg} bottom row), where for initialization we take $c_1 = 1$ and $c_2 = 0$. Similar to \cite{Yang}, in these experiments we choose parameters $\eta = 0.1$, $\lambda_1 = 0.65$, $\lambda_2 = 1$ and time step $\tau = 0.001$. A two-stage procedure is used where we first solve \eqref{dis:Seg} with a large value of $\eps = 80$, and once a quasi-steady state has been achieved (approximately at $n = 5000$), we reduce the value of $\eps$ to $0.01$ and resume the iterative process until a new steady state is reached.  In Figure~\ref{fig:seg} we display the $1/2$-level set of discrete solution $\vp^n$ in red overlayed with the original image at iterations $n = 1000, 3000$ and $10000$, as well as the discrete solution itself at $n = 10000$. We note that the segmentation task is well-performed in both experiments, and for a comparison between the Cahn--Hilliard approach \eqref{Seg} with the classical Chan--Vese approach we refer the reader to \cite{Yang}.

\subsection{Image inpainting}
In \cite{Bert2} the authors proposed the following Cahn--Hilliard model to restore missing or damaged details in a measurable subdomain $D \subset \Omega$ for a binary image function $I: \Omega \to \{-1,1\}$:
\begin{align}\label{Inpaint}
\pd_t \vp = \Lap \mu + \lambda(x)(I(x) - \vp), \quad \mu = - \eps \Lap \vp + \eps^{-1} F'(\vp), \quad \lambda(x) = \lambda_0 \chi_{\Omega \setminus D}(x),
\end{align}
furnished with homogeneous Neumann conditions. In the above $\chi_A(x)$ denotes the characteristic function of the set $A$, and the fidelity term $\lambda(x)(I(x) - \vp)$ with large constant $\lambda_0$ demands the solution $\vp$ to be close to the data $I(x)$ in the undamaged region $\Omega \setminus D$. In contrast to \eqref{CHO} there is no analogue to \eqref{CHO:mass} for the mean value $\mean{\vp}$ due to the function $\lambda(x) = \lambda_0 \chi_{\Omega \setminus D}(x)$ (see \cite{CFM} for an inequality estimate) and thus our stability analysis does not apply to a standard SAV discretization of \eqref{Inpaint}.   On account of the purpose of the fidelity term, a reasonable modification of \eqref{Inpaint} is 
\begin{align}\label{Inpaint:mod}
\pd_t \vp = \Lap \mu + \lambda(x) (I(x) - \mathcal{T}(\vp)) , \quad \mathcal{T}(\vp) = \max(-1, \min(1,\vp)).
\end{align}

For the modified equation \eqref{Inpaint:mod}, the source term is bounded and so a SAV-based discretization reads as
\begin{equation}\label{dis:Inpaint}
\begin{aligned}
\vp^n - \vp^{n-1} & = \tau \Lap \mu^n + \tau \lambda(x) (I(x) - \mathcal{T}(\vp^{n-1})), \\
\mu^{n} & = - \eps \Lap \vp^{n} + \frac{q^{n}}{\eps Q(\vp^{n-1})} F'(\vp^{n-1}),  \\
q^{n} - q^{n-1} & = \frac{1}{2 \eps Q(\vp^{n-1})} (F'(\vp^{n-1}), \vp^{n} - \vp^{n-1}),
\end{aligned}
\end{equation}
furnished with homogeneous Neumann conditions.  It is straightforward to infer the estimate \eqref{stab:est:CHO} for $G^k$, leading to stability of the SAV scheme for sufficiently small $\tau$.

We apply \eqref{dis:Inpaint} to perform inpainting of a double stripe shown in Figure~\ref{fig:inpaint}(a). Similar to image segmentation, we adopt a two-stage procedure where we first solve \eqref{dis:Inpaint} with parameters $\lambda_0 = 10$, $\eps = 100$, $\tau =0.1$, and after $n = 3000$ iterations we then switch a new set of parameters $\lambda_0 = 0.1$, $\eps = 5$, $\tau = 1$. Figure~\ref{fig:inpaint}(d) displays the discrete solution $\vp^n$ at $n = 4000$ and is an acceptable reconstruction result, cf.~\cite{Kim}.

\subsection{Tumor growth dynamics}\label{sec:Tumor}
In \cite{Crisbook,GLSS,Wise} the following model is used to describe tumor growth dynamics:
\begin{subequations}\label{Tumor}
\begin{alignat}{2}
\pd_t \vp & = \div (m(\vp) \nabla \mu) + P(\vp, \mu, \sigma), \\
\mu&  = - \eps \Lap \vp + \eps^{-1} F'(\vp) - \chi \sigma, \\
\pd_t \sigma & = \div (n(\vp) \nabla ( \chi_\sigma  \sigma - \eta  \vp)) - S(\vp, \mu, \sigma).
\end{alignat}
\end{subequations}
In the above, $\vp \in [-1,1]$ represents the difference in volume fraction between tumor and healthy tissues; $\mu$ is the associated chemical potential; $\sigma$ denotes the concentration of a nutrient for the tumor; $m(\vp)$ and $n(\vp)$ denote positive mobilities for $\vp$ and $\sigma$, \tc{bounded above and below by constants $m_1, m_0$ and $n_1, n_0$, respectively}; parameters $\chi_\sigma > 0$, $\chi > 0$ and $\eta \geq 0$ model the nutrient diffusivity, chemotaxis and active transport mechanisms respectively; $P$ accounts for tumor proliferation and apoptosis; $S$ models nutrient consumption. When $\chi = \eta$, and furnishing with homogeneous Neumann conditions, the model admits the following energy identity
\begin{align*}
& \frac{d}{dt} \int_\Omega \frac{\eps}{2} |\nabla \vp|^2 + \frac{1}{\eps} F(\vp) + \frac{1}{2} |\sigma|^2 - \chi \vp \sigma \dx + \int_\Omega m(\vp) |\nabla \mu|^2 + n(\vp)|\nabla (\chi_\sigma \sigma - \chi \vp)|^2 \dx \\
& \quad = \int_\Omega P(\vp, \mu, \sigma) \mu - S(\vp, \mu, \sigma)(\sigma - \eta \vp) \dx.
\end{align*}
The forms of the source terms $P$ and $S$ determine whether $E(\vp, \sigma) = \int_\Omega \frac{\eps}{2} |\nabla \vp|^2 + \frac{1}{\eps} F(\vp) + \frac{1}{2} |\sigma|^2 - \chi \vp \sigma \dx$ is a Lyapunov functional. One example proposed in \cite{Hawkins} is to consider $P(\vp, \mu, \sigma) = S(\vp, \mu, \sigma) = Q(\vp)(\sigma - \chi \vp - \mu) = Q(\vp) ( \frac{\delta E}{\delta \sigma} - \frac{\delta E}{\delta \vp})$ as a difference of chemical potentials weighted by a bounded, nonnegative function $Q$. This choice would yield $\frac{d}{dt} E(\vp, \sigma) \leq 0$, and the standard SAV schemes can be applied, see \cite{SWZ}.  Another, more phenomenological in nature, is the following used in e.g.~\cite{GLSS,Wise}:
\begin{align}\label{linearkinetics}
P(\vp, \sigma) = h(\vp)(\mathcal{P}k(\sigma) - \mathcal{A}), \quad S(\vp, \sigma) = \mathcal{C} h(\vp) \sigma,
\end{align}
with constant proliferation rate $\mathcal{P}$, apoptosis rate $\mathcal{A}$, nutrient consumption rate $\mathcal{C}$ and bounded, nonnegative functions $h$ and $k$. One example is $h(s) = \max(-1, \min(1, \tfrac{1}{2}(1+s)))$ and $k(s) = \max(0,\min(\sigma_\infty,s))$ with fixed external supply concentration $\sigma_\infty$.

We consider a reformulation of the model with a new variable $\hat{\mu} := \mu + \chi \sigma$, where a SAV-based time discretization of \eqref{Tumor} with source terms of the form \eqref{linearkinetics} is
\begin{subequations}\label{Tumor:dis}
\begin{alignat}{2}
\sigma^{n} - \sigma^{n-1} & = \tau \, \div (n(\vp^{n-1}) \nabla (\chi_\sigma \sigma^{n} - \eta \vp^{n-1}) \tc{)} -  \tau S(\vp^{n-1}, \sigma^{n}),  \label{Tumor:dis:4} \\
\vp^{n} - \vp^{n-1} & = \tau\, \div (m(\vp^{n-1}) \nabla (\hat{\mu}^{n} - \chi \sigma^n)) + \tau P(\vp^{n-1}, \sigma^{n}) \label{Tumor:dis:1} \\
\hat{\mu}^{n} & = - \eps \Lap \vp^{n} + \frac{q^{n}}{\eps Q(\vp^{n-1})} F'(\vp^{n-1}), \label{Tumor:dis:2} \\
q^{n} - q^{n-1} & = \frac{1}{2 \eps Q(\vp^{n-1})} (F'(\vp^{n-1}), \vp^{n} - \vp^{n-1}), \label{Tumor:dis:3}
\end{alignat}
\end{subequations}
furnished with homogeneous Neumann conditions. Note that we allow $\eta$ and $\chi$ to be different constants. Given $(\vp^{n-1}, q^{n-1}, \sigma^{n-1})$, we first solve for $\sigma^n$ with \eqref{Tumor:dis:4}, and then solve for $(\vp^n, \hat{\mu}^n, q^n)$ with \eqref{Tumor:dis:1}-\eqref{Tumor:dis:3}.  As $P(\vp,\sigma)$ is bounded and $|S(\vp, \sigma)| \leq C|\sigma|$, by testing \eqref{Tumor:dis:1} with $\hat{\mu}^n$ and $\vp^n$, \eqref{Tumor:dis:2} with $\vp^n - \vp^{n-1}$, \eqref{Tumor:dis:3} with $2q^n$ and \eqref{Tumor:dis:4} with $K\sigma^n$ for some constant $K>0$ yet to be specified, upon summing we obtain for
\[
G^k := \frac{1}{2} \| \vp^k \|^2 + \frac{\eps}{2} \| \nabla \vp^k \|^2 + |q^k |^2 + \frac{K}{2} \| \sigma^k \|^2,
\]
the estimate
\begin{align*}
& G^n - G^{n-1} + \frac{1}{2} \big ( \| \vp^n - \vp^{n-1} \|^2 + \eps \| \nabla(\vp^n - \vp^{n-1}) \|^2 + K \| \sigma^n - \sigma^{n-1}\|^2 \big ) \\
& \qquad +  |q^n - q^{n-1}|^2 \ + \tau \int_\Omega m(\vp^{n-1}) |\nabla \hat{\mu}^n|^2 + K \chi_\sigma n(\vp^{n-1}) |\nabla \sigma^n|^2 \dx \\
& \quad = \tau (P(\vp^{n-1}, \sigma^{n-1}), \hat{\mu}^n + \vp^n) - \tau (m(\vp^{n-1}) \nabla \hat{\mu}^n, \nabla (\vp^n - \chi \sigma^n)) \\
& \qquad + \tau(\nabla \sigma^n, K\eta n(\vp^{n-1}) \nabla \vp^{n-1} + \chi m(\vp^{n-1}) \nabla \vp^n) - K\tau (S(\vp^{n-1}, \sigma^{n}), \sigma^n) \\
& \quad \leq C \tau \| \hat{\mu}^n \|_{L^1} + \frac{\tau m_0}{4} \| \nabla \hat{\mu}^n \|^2 +  \tau \Big (\frac{K \chi_\sigma n_0}{4} + c \Big ) \| \nabla \sigma^n \|^2 \\
& \qquad + C \tau \big ( \| \vp^n \|_{H^1}^2 + \| \nabla \vp^{n-1} \|^2  + \| \sigma^n \|^2 \big ) ,
\end{align*}
where $c> 0$ is independent of $K$.  Choosing $K$ sufficiently large so that $c < \frac{K \chi_\sigma n_0}{4}$, and using \eqref{mu:trick} and \eqref{mean:mu:est} for $\| \hat \mu^n \|_{L^1}$ we arrive at 
\begin{align*}
& G^n - G^{n-1} + \frac{\min(1,\eps)}{2} \| \vp^n - \vp^{n-1} \|_{H^1}^2 + \frac{K}{2} \| \sigma^n - \sigma^{n-1}\|^2 \\
& \qquad +  |q^n - q^{n-1}|^2 \ + \frac{\tau m_0}{2} \|\nabla \hat{\mu}^n \|^2 + \frac{\tau K \chi_\sigma n_0}{2} \|\nabla \sigma^n\|^2 \\
& \quad \leq C \tau \big ( 1 + G^n + G^{n-1} \big ),
\end{align*} 
which ensures the stability of the SAV scheme for sufficiently small $\tau$.

\begin{remark}
If $\eta = 0$ and $\sigma^0$ is bounded and non-negative, then by applying a comparison principle to \eqref{Tumor:dis:4} we can deduce that $\sigma^n$ is also bounded and non-negative for all $n$.  Thus, we can replace $k(\sigma)$ in \eqref{linearkinetics} with simply $\sigma$.
\end{remark}

We consider a similar setting to \cite{GaTr} where we set $\chi_\sigma = 25$, $\eta = \chi = 5$, constant mobilities $n(\vp) = m(\vp) = 1$, $\eps = 0.01$, $\tau = 0.001$, with initial conditions $\sigma^0 = 1$ and 
\[
\vp^0(\bm{x}) = - \tanh \Big ( \frac{|\bm{x}| - (0.05 +0.02\cos (2 \theta))}{\sqrt{2} \times 0.01} \Big), \quad \bm{x} = |\bm{x}|(\cos \theta, \sin \theta)^{\top}.
\]
\tc{Such a choice for the initial condition $\vp^0$ means the tumor is initially bounded by} a slightly perturbed circle. The source functions are chosen as $P(\vp, \sigma) = S(\vp,\sigma) = \frac{1}{2}k(\sigma)(1+\varphi)$ with $\sigma_\infty = 1$. In Figure~\ref{fig:tumor} we display the discrete solution $\vp^n$ at $n = 0$ (initial condition), $n = 8000$, $n = 13000$ and $n = 18000$, where the tumor undergoes morphological instabilities and develop fingers towards regions of high nutrient concentration.  Not shown here are plots of the discrete solution $\sigma^n$, which are visually similar to Figure~\ref{fig:tumor}, \tc{namely the nutrient concentration $\sigma^n$ is lower in the tumor region $\{\vp^n = 1\}$ and higher in the tissue region $\{\vp^n = -1\}$}. We note that these seem to \tc{resemble} previous simulations found in \cite{GLSS,GaTr,Wise}.

\section{Conclusion}
In this work we investigate the stability of a time discretization based on the relaxed scalar auxiliary variable (RSAV) approach of \cite{RSAV} for Cahn--Hilliard systems with bounded mass source. In general these systems may not exhibit dissipative structures, and so the stability of such SAV-based schemes are not immediate. Our proofs rely on a key estimate of a product term between the scalar auxiliary variable and the cubic nonlinearity.  \tc{For the} convergence of time discrete solutions \tc{our analysis covers two choices of the interpolating parameter $\zeta_n^\tau$ in \eqref{dis:4}.} Numerical simulations are supportive of our results and we are able to replicate the expected solution behavior for Cahn--Hilliard systems in tumor growth, image segmentation and image inpainting. 

\begin{figure}[htbp]
\centering
\subfloat[$n=0$]{
\includegraphics[width=3.5cm]{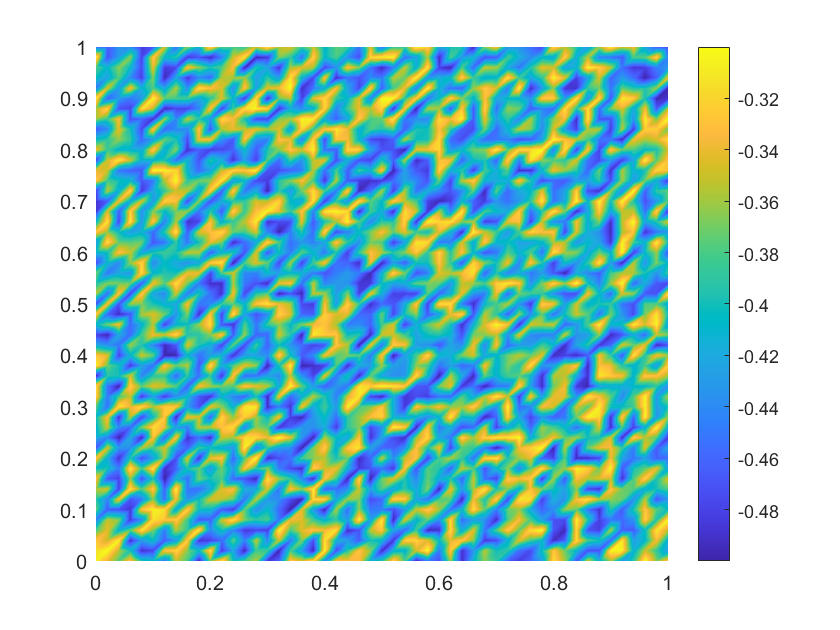}
}%
\subfloat[$n=500$]{
\includegraphics[width=3.5cm]{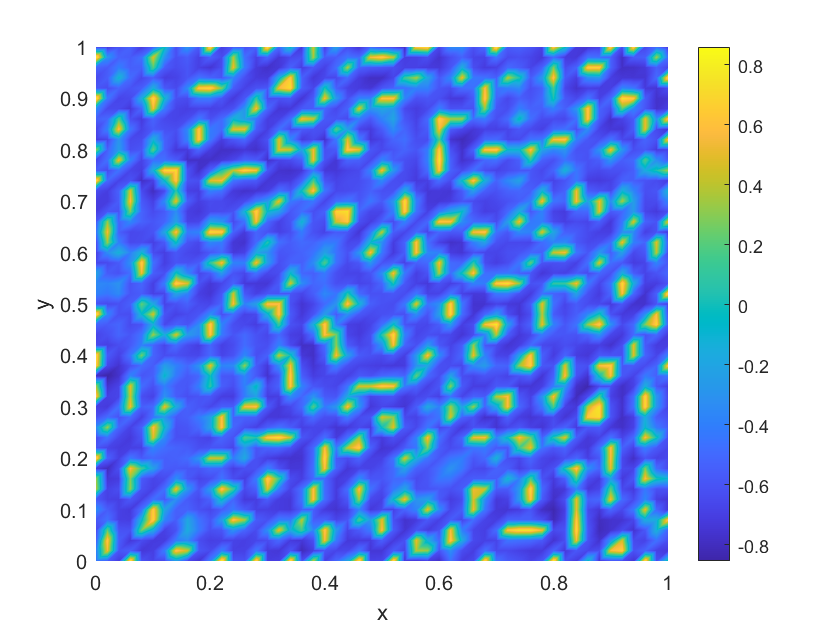}
}%
\subfloat[$n=10000$]{
\includegraphics[width=3.5cm]{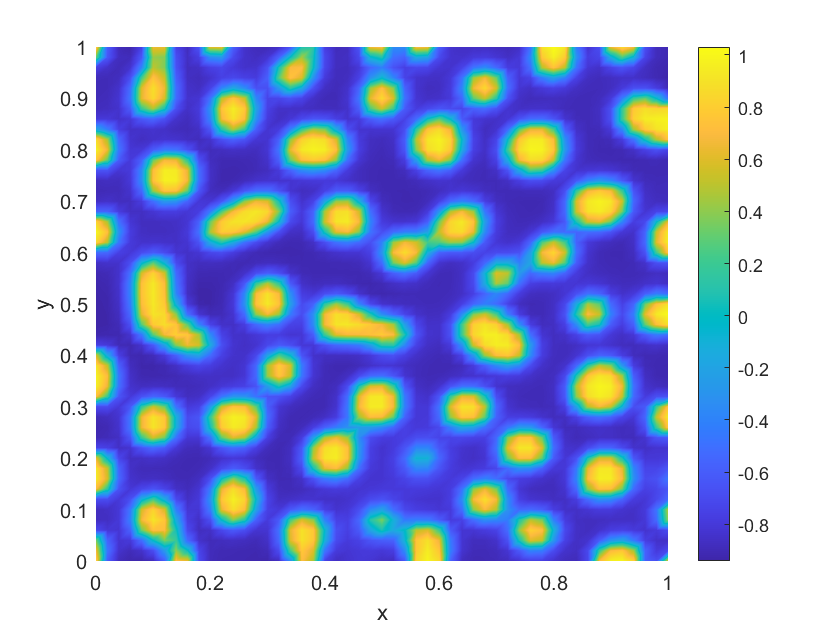}
}%
\subfloat[$n=50000$]{
\includegraphics[width=3.5cm]{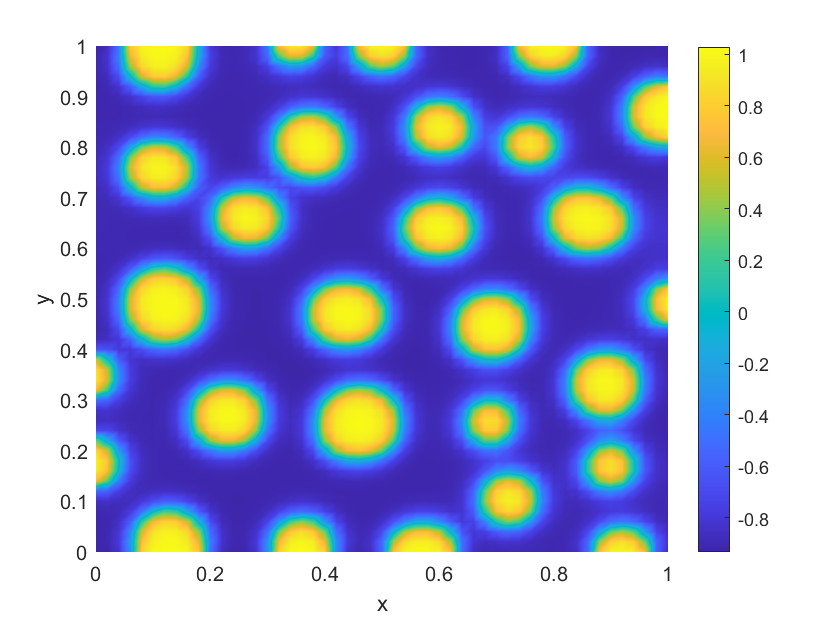}
}%
\\
\subfloat[$n=0$]{
\includegraphics[width=3.5cm]{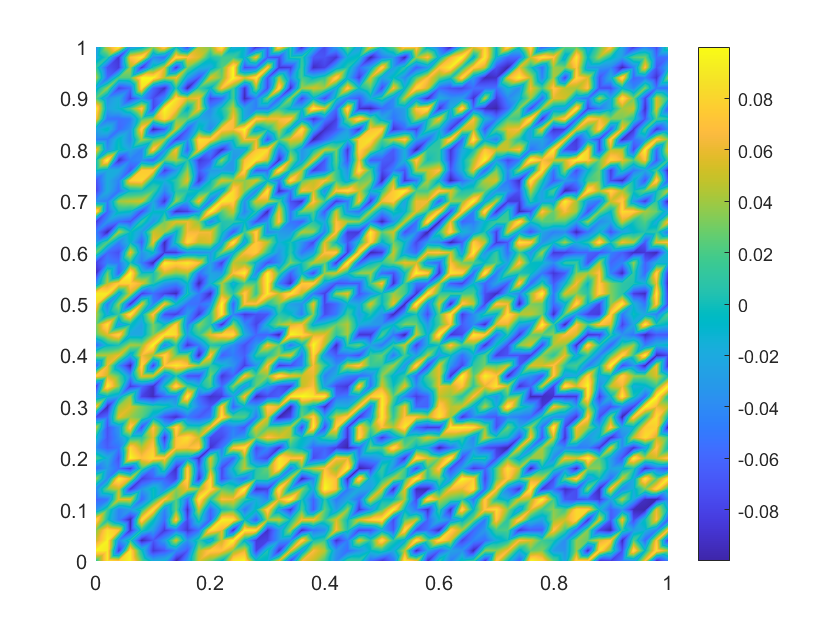}
}
\subfloat[$n=500$]{
\includegraphics[width=3.5cm]{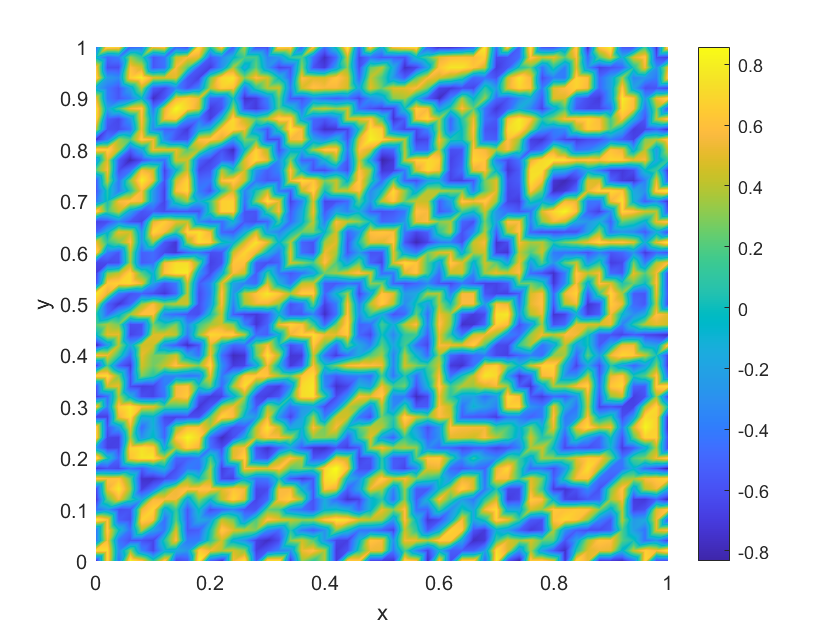}
} 
\subfloat[$n=10000$]{
\includegraphics[width=3.5cm]{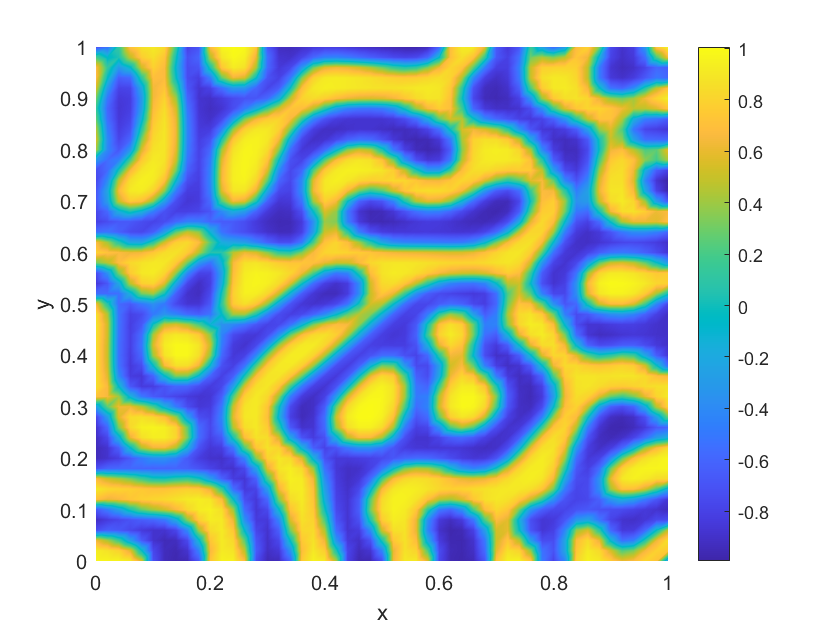}
} 
\subfloat[$n=50000$]{
\includegraphics[width=3.5cm]{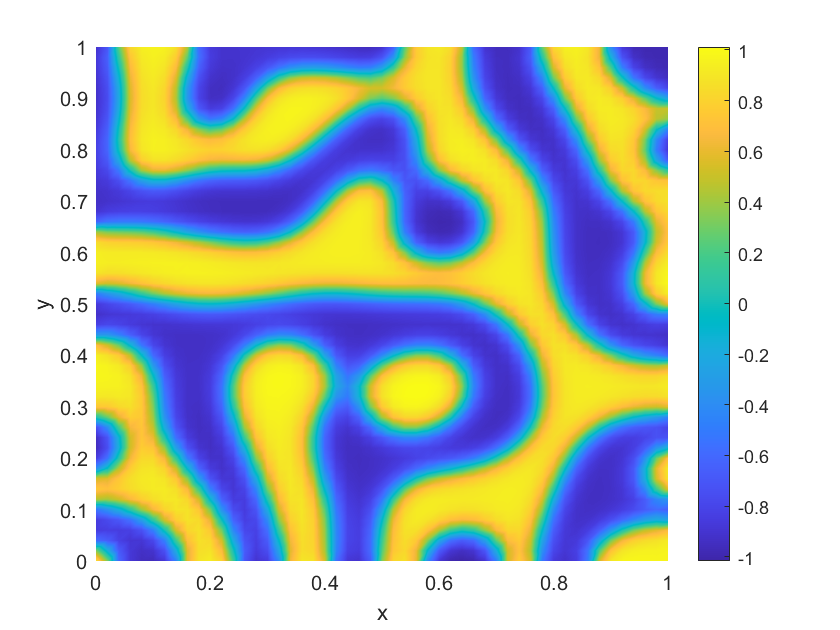}
}
\caption{Simulation of diblock copolymer dynamics with \eqref{CHO:diss}. Top row  (subfigures (a), (b), (c), (d)) display the evolution starting with initial condition $\vp^0 = -0.5 + U[0,0.2]$. Bottom row (subfigures (e), (f), (g), (h)) display the evolution starting with initial condition $\vp^0 = -0.1 + U[0,0.2]$.}
\label{fig:CHO1}
\end{figure}

\begin{figure}[htbp]
\centering
\subfloat[$n=1000$]{
\includegraphics[width=3.5cm]{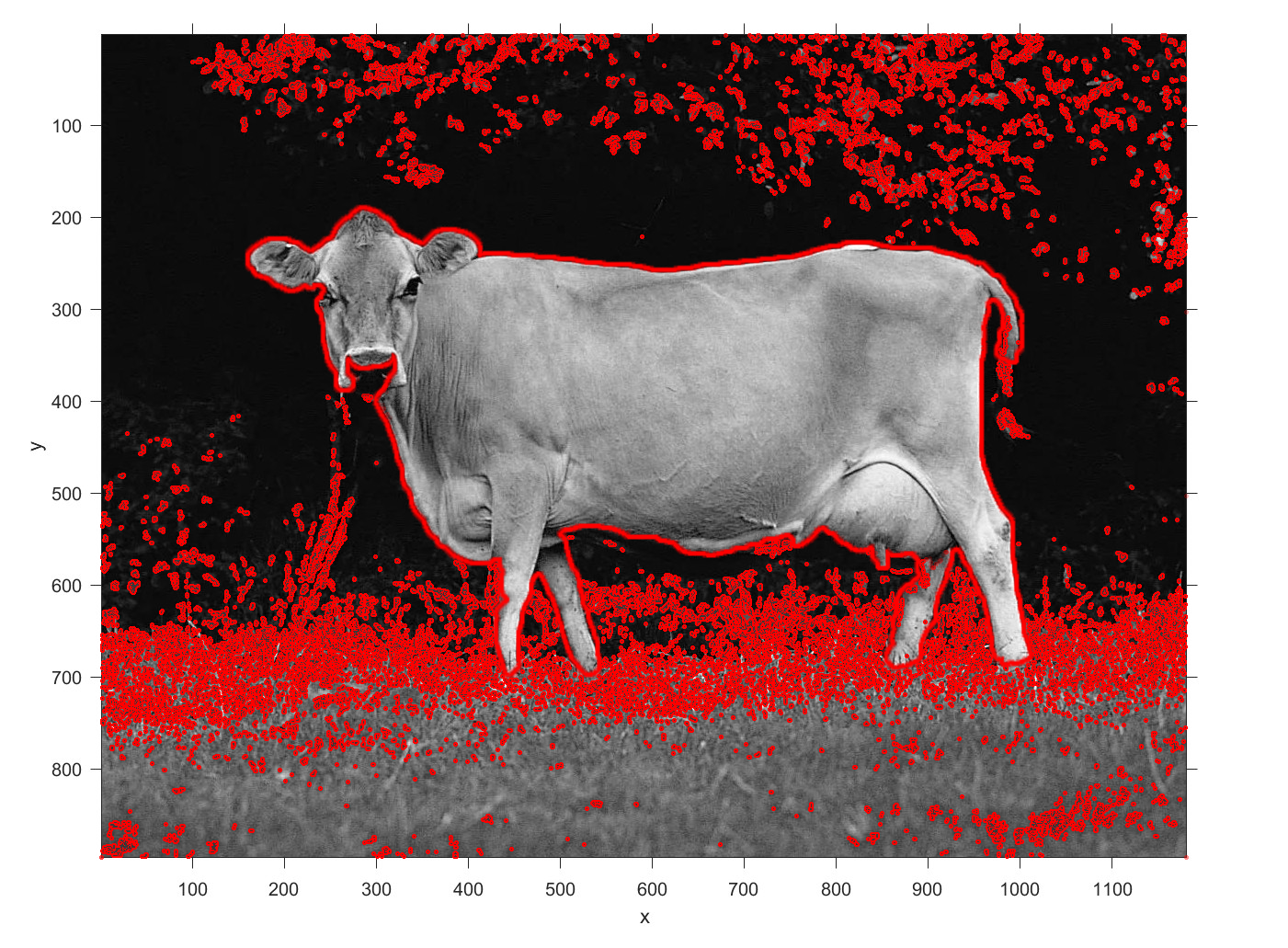}
}%
\subfloat[$n=3000$]{
\includegraphics[width=3.5cm]{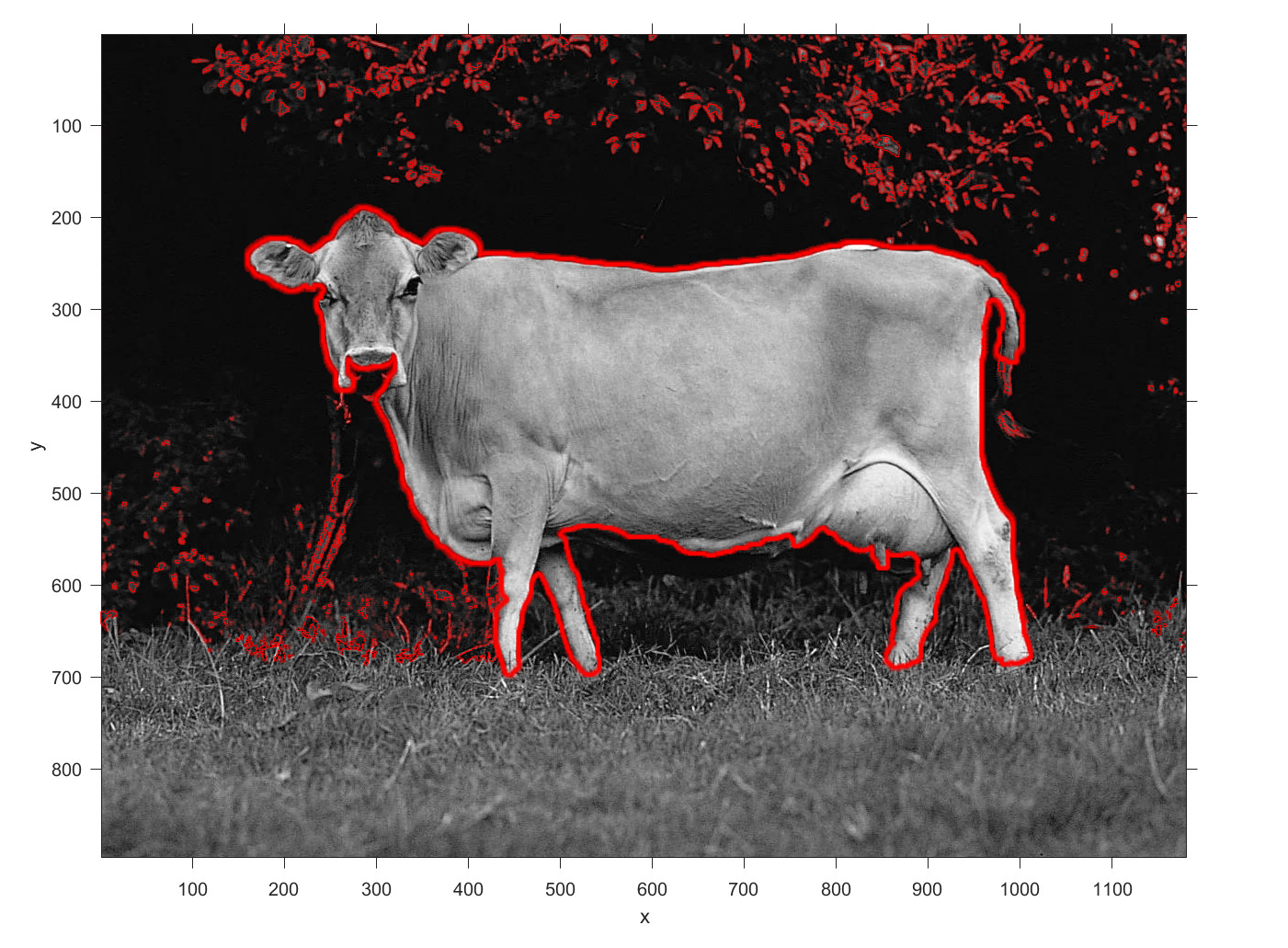}
}%
\subfloat[$n=10000$]{
\includegraphics[width=3.5cm]{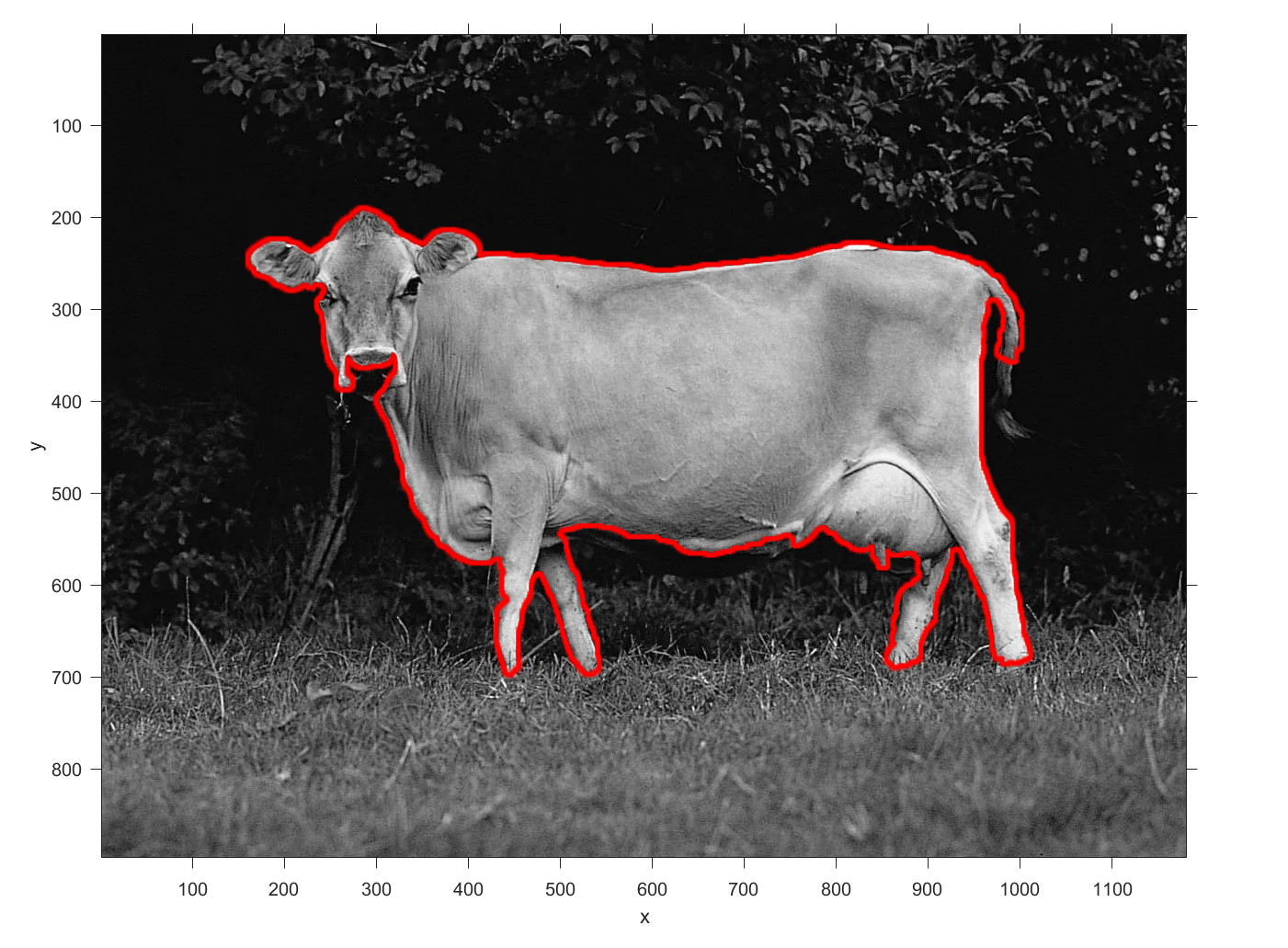}
}%
\subfloat[$n=10000$]{
\includegraphics[width=3.5cm]{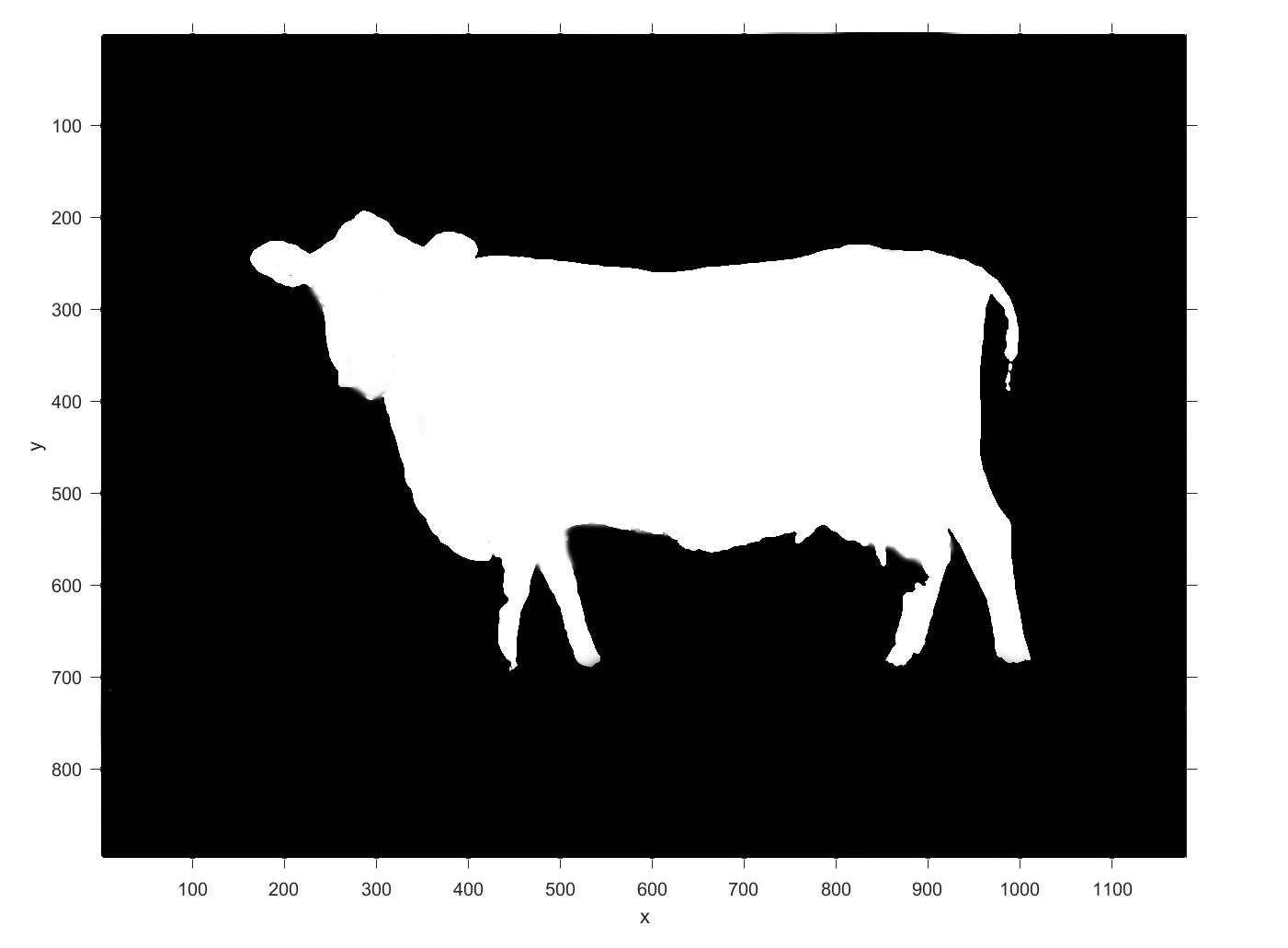}
}
\\
\subfloat[$n=1000$]{
\includegraphics[width=3.5cm]{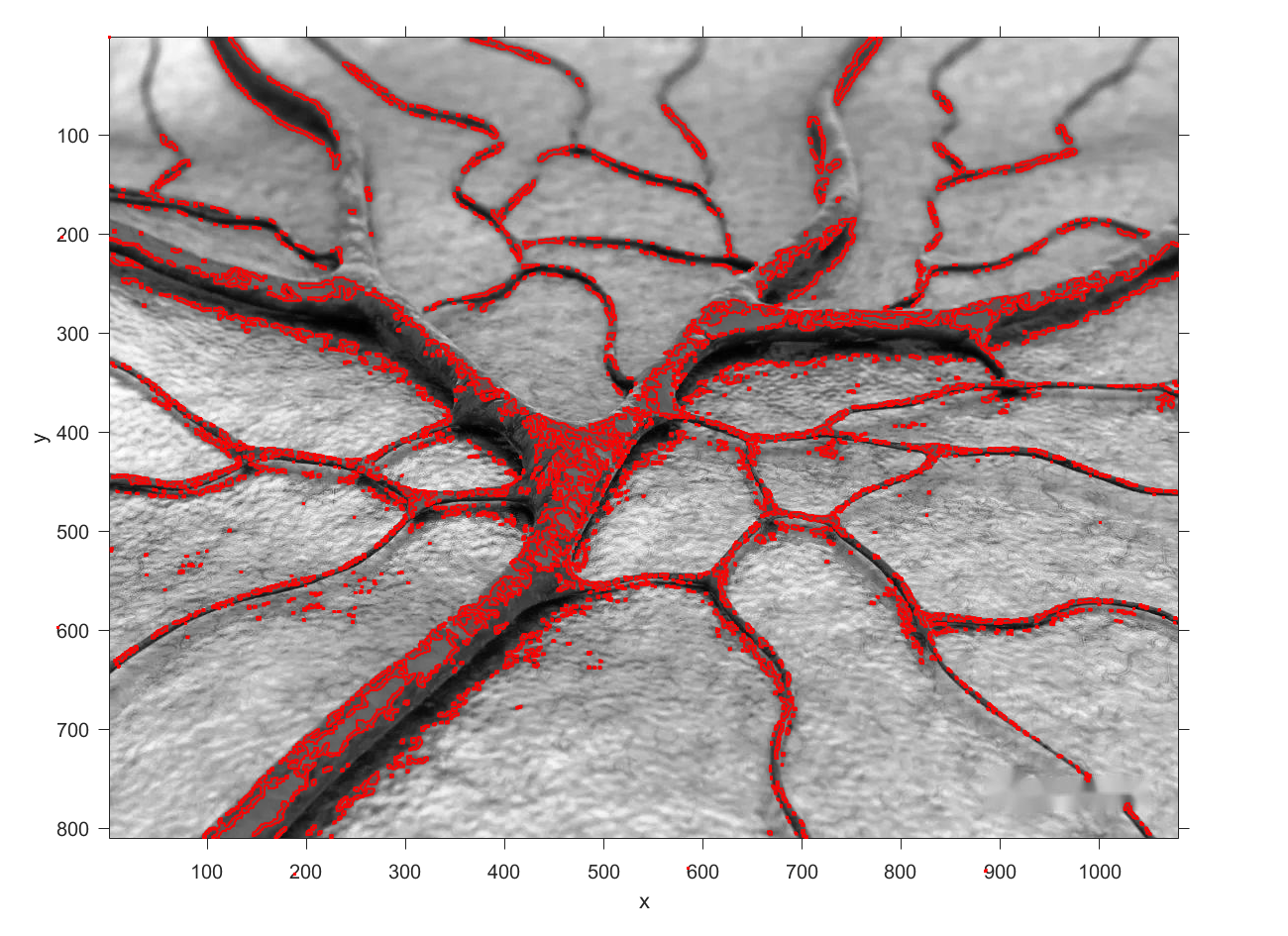}
}%
\subfloat[$n=3000$]{
\includegraphics[width=3.5cm]{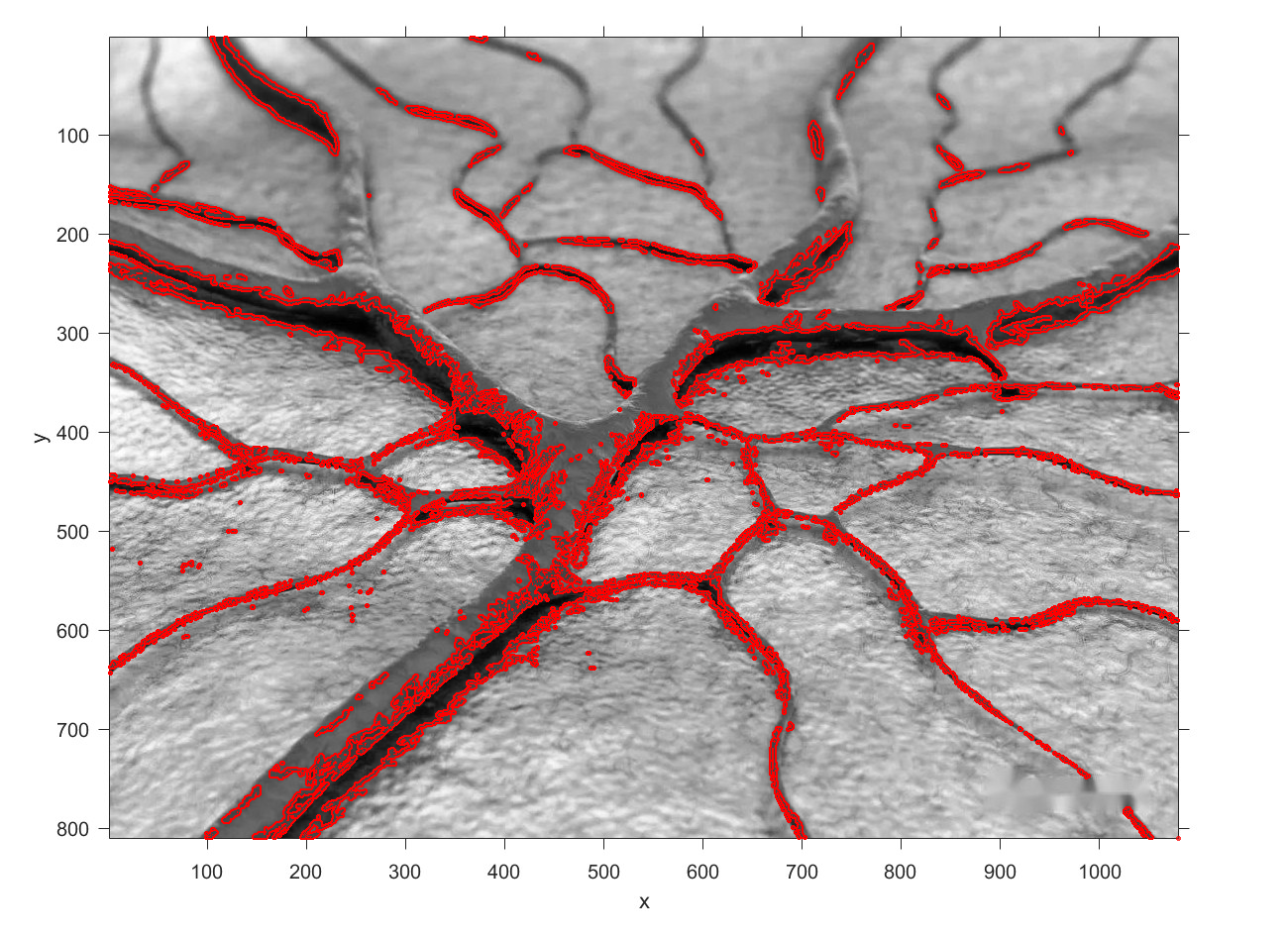}
}%
\subfloat[$n=10000$]{
\includegraphics[width=3.5cm]{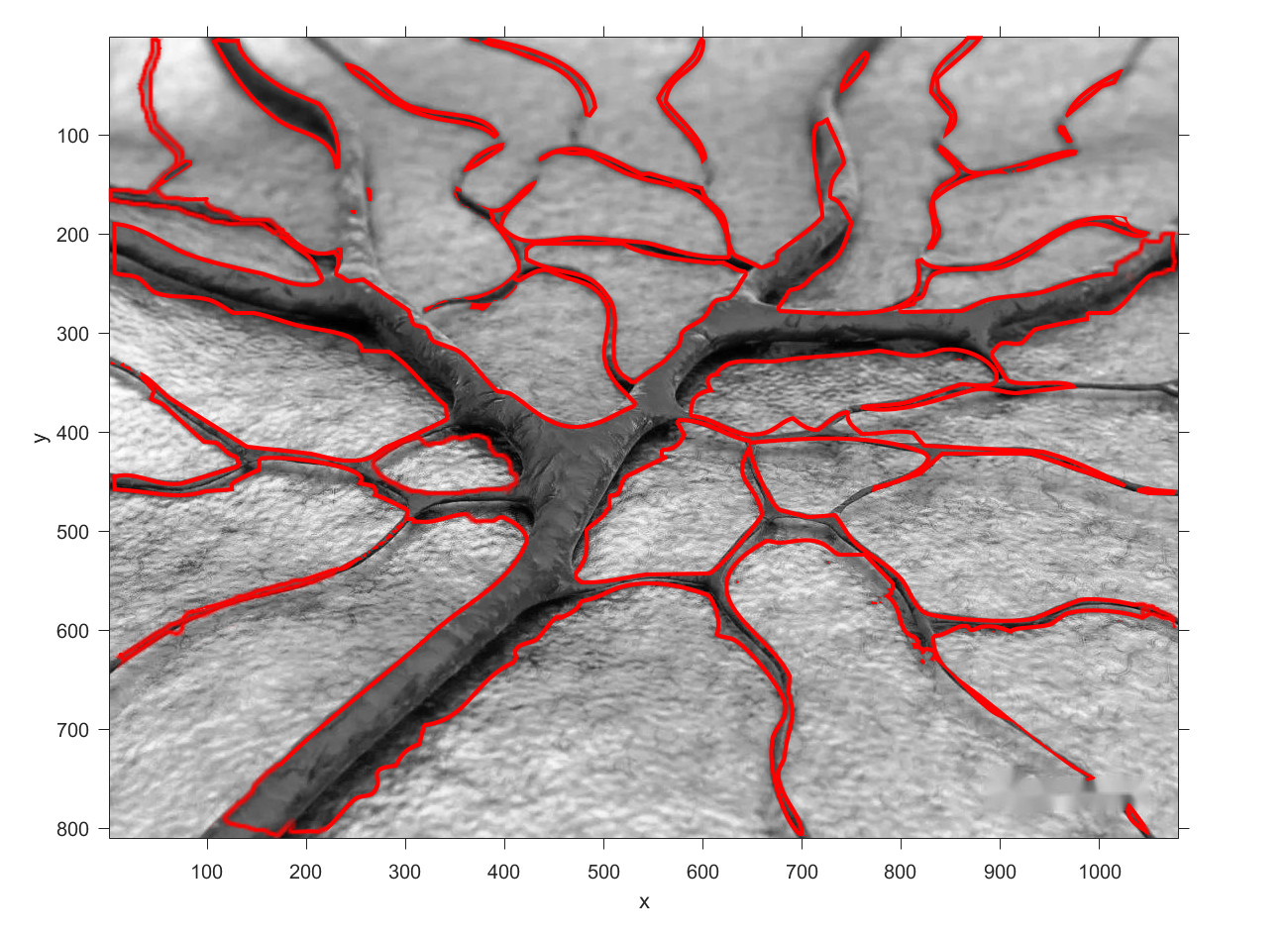}
}%
\subfloat[$n=10000$]{
\includegraphics[width=3.5cm]{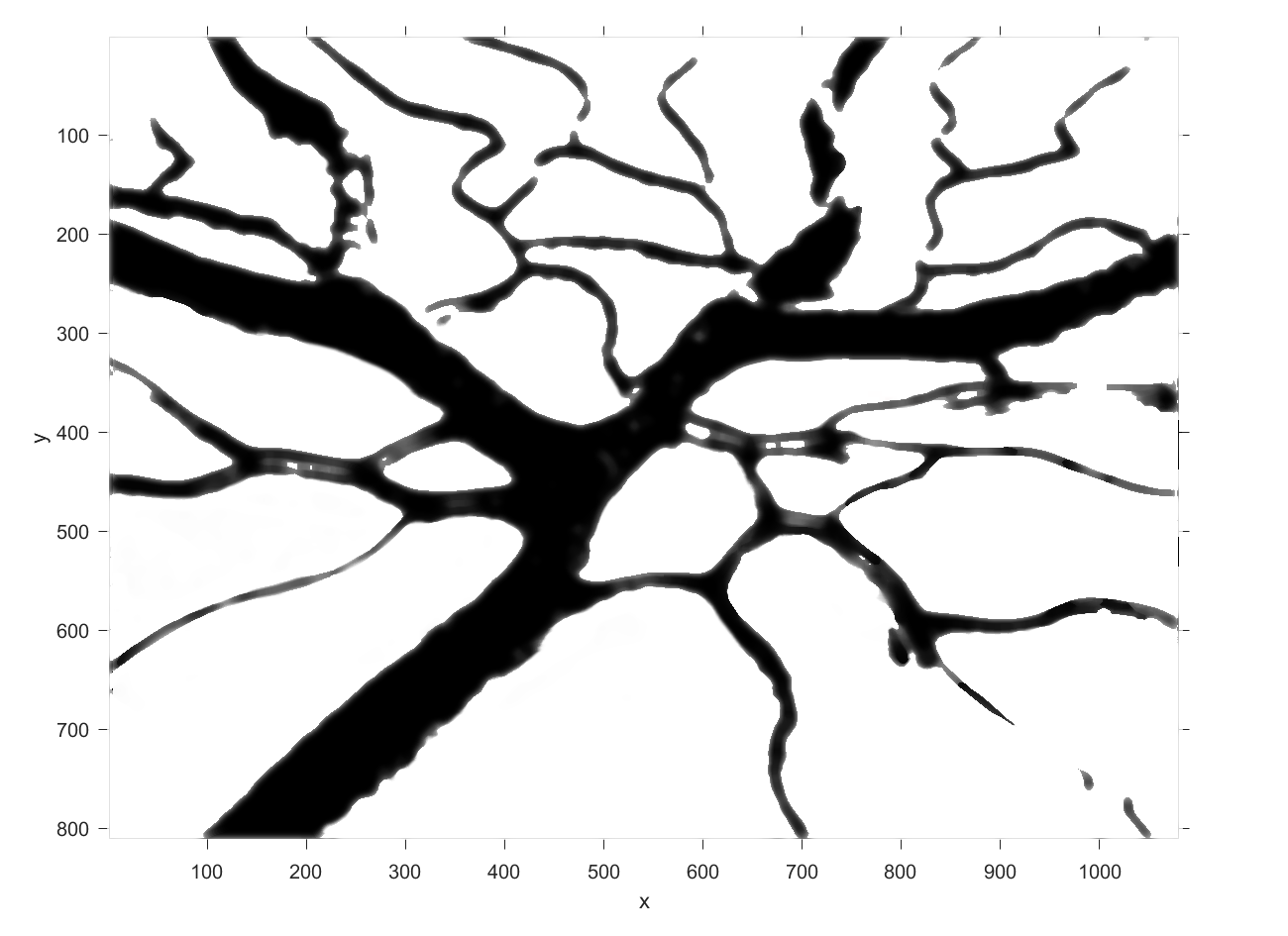}
}%
\caption{Image segmentation with \eqref{dis:Seg}. Segmentation of a cow image (top row) and a blood vessel image (bottom row) with \eqref{dis:Seg}. The first three subfigures ((a), (b), (c) and (e), (f), (g)) of each row display the 1/2-level set of $\vp^n$ (colored red) overlayed with the original image at $n = 1000$, $n = 3000$ and $n = 10000$, respectively.  The last subfigure ((d) and (h)) of each row display the discrete solution $\vp^n$ at $n = 10000$.}
\label{fig:seg}
\end{figure}
\begin{figure}[htbp]
\centering
\subfloat[$n = 0$]{
\includegraphics[width=3.5cm]{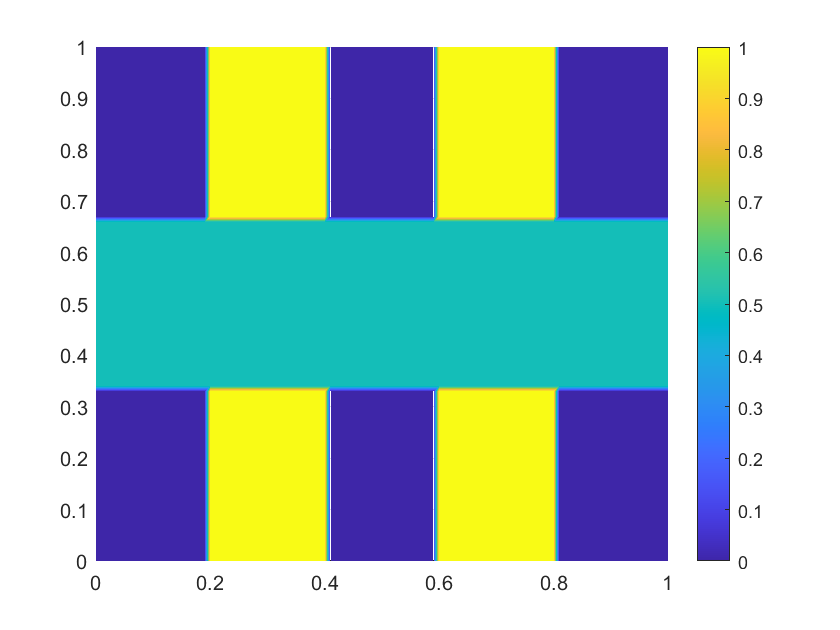}
}%
\subfloat[$n=400$]{
\includegraphics[width=3.5cm]{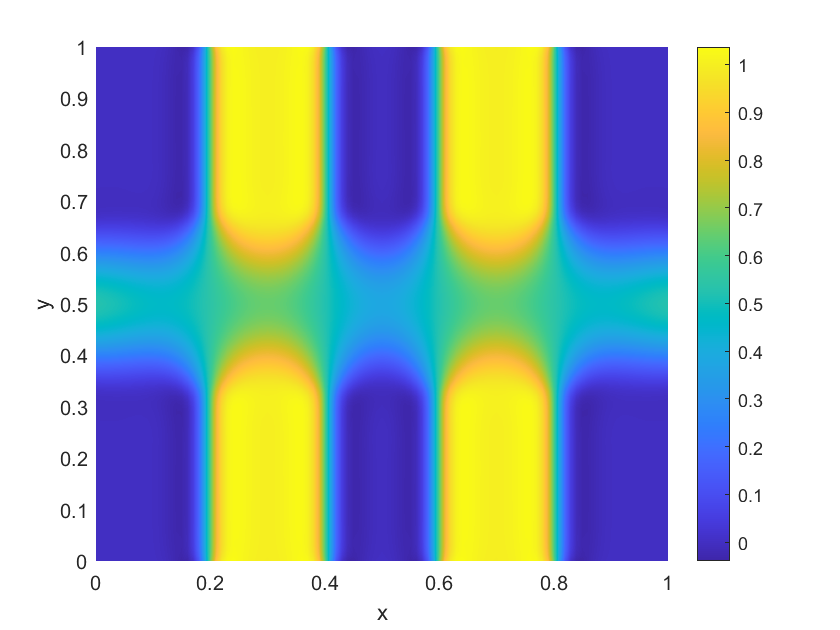}
}
\subfloat[$n=3200$]{
\includegraphics[width=3.5cm]{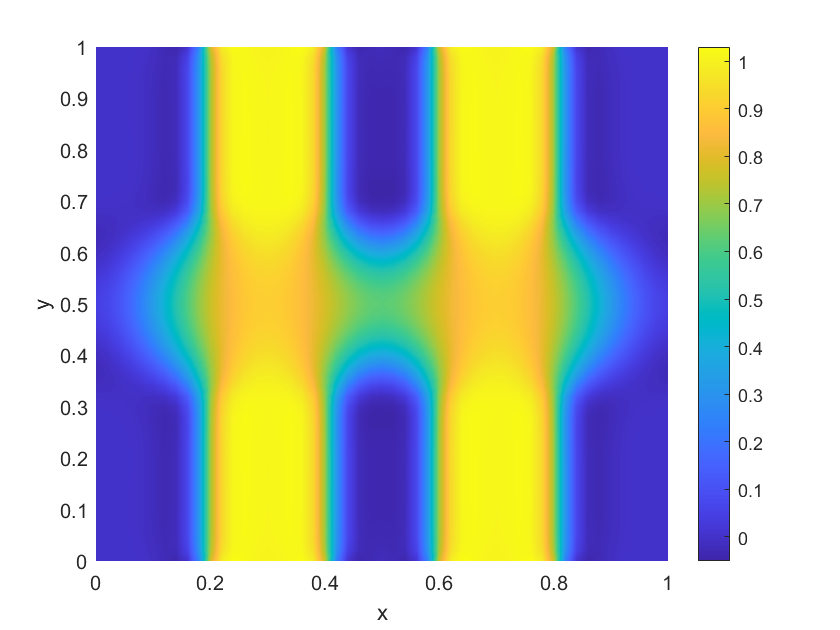}
}%
\subfloat[$n=4000$]{
\includegraphics[width=3.5cm]{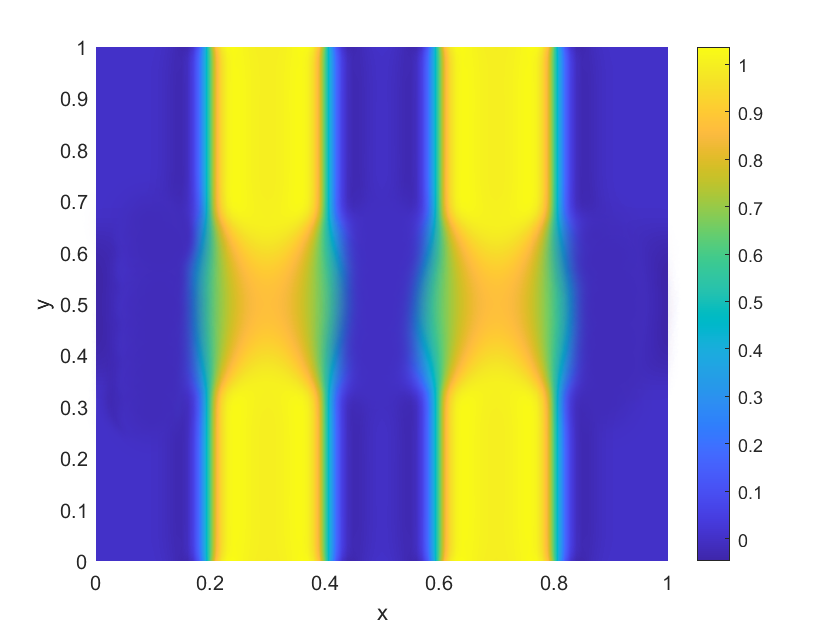}
}%
\caption{Inpainting of a double stripe binary image with \eqref{dis:Inpaint}.}
\label{fig:inpaint}
\end{figure}
\begin{figure}[htbp]
\centering
\subfloat[$n=0$]{
\includegraphics[width=3.5cm]{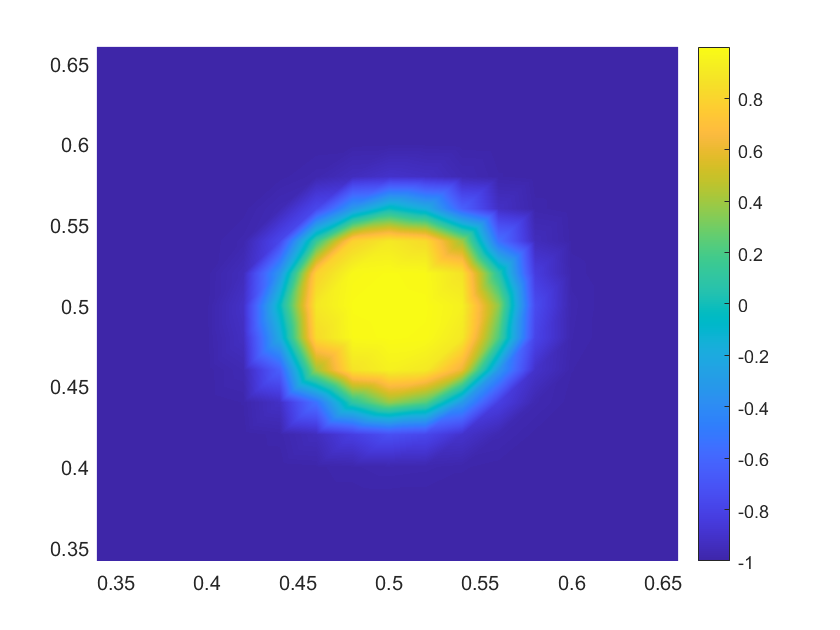}
}%
\subfloat[$n=8000$]{
\includegraphics[width=3.5cm]{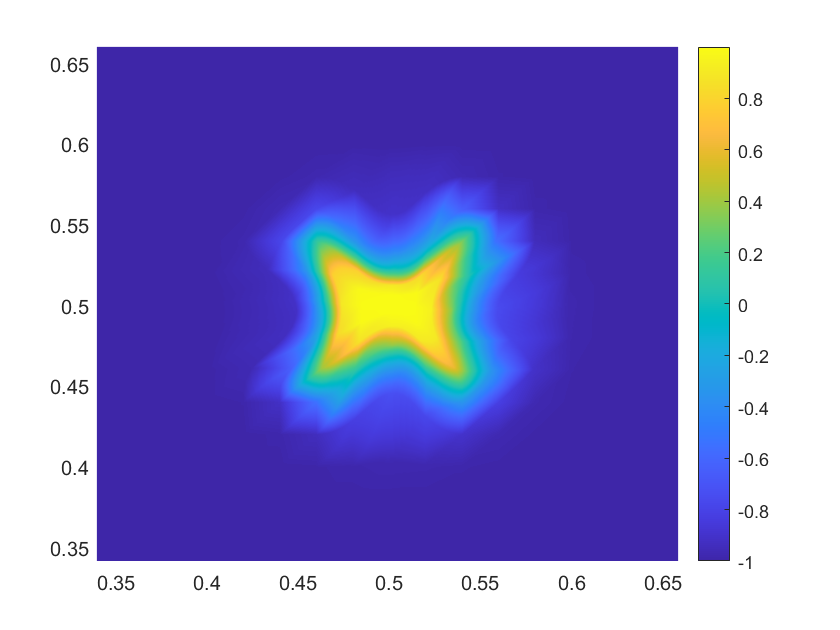}
}%
\subfloat[$n=13000$]{
\includegraphics[width=3.5cm]{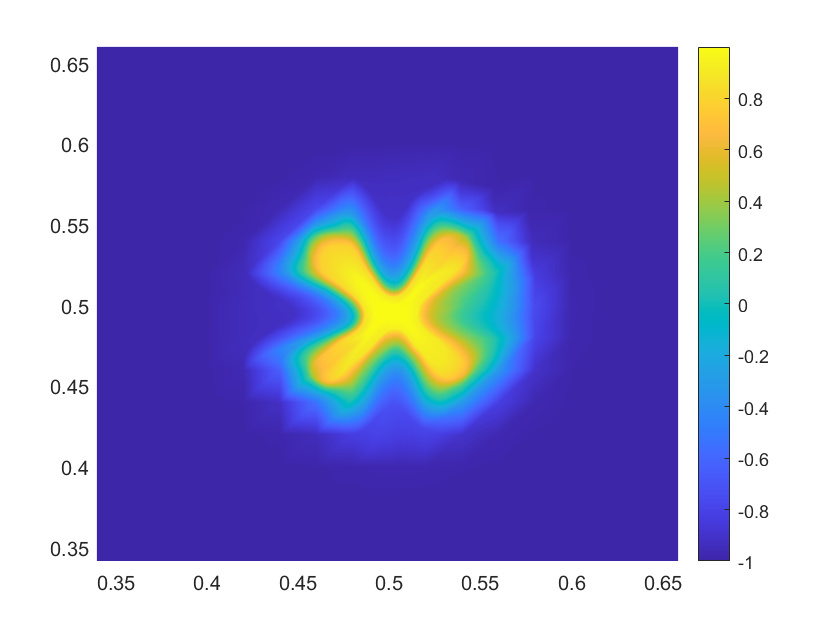}
}%
\subfloat[$n=18000$]{
\includegraphics[width=3.5cm]{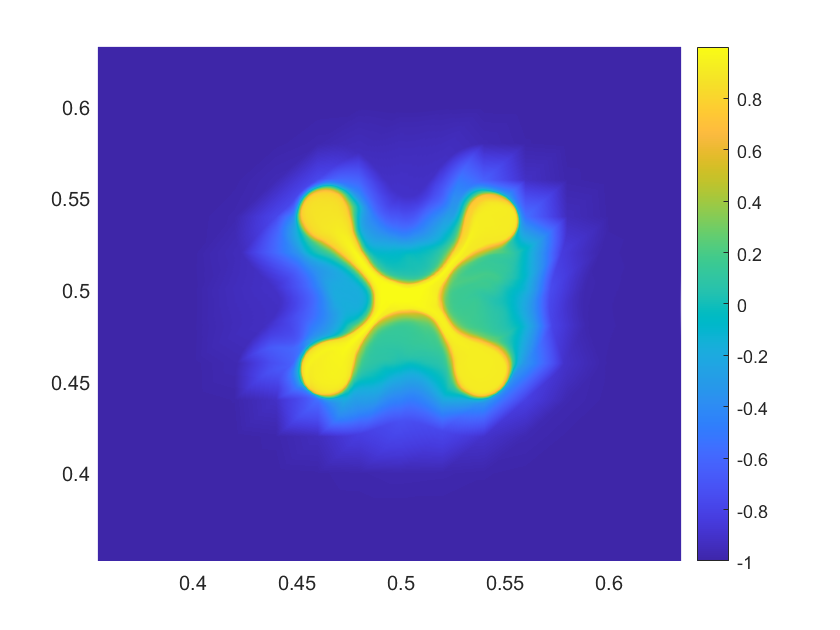}
}%
\caption{Simulation of tumor growth dynamics with \eqref{Tumor:dis}.}
\label{fig:tumor}
\end{figure}

\section*{Acknowledgments}
The authors gratefully acknowledge the support by the Research Grants Council of the Hong Kong Special Administrative Region, China [Project No.: HKBU 14302319] and Hong Kong Baptist University [Project No.: RC-OFSGT2/20-21/SCI/006].


\footnotesize
\bibliographystyle{plain}

\end{document}